\documentclass{amsart}
\usepackage{amssymb}
\usepackage{amsfonts}
\usepackage{amsmath}

\newtheorem {theorem}{Theorem}

\newtheorem {corollary}{Corollary}

\newtheorem {definition}{Definition}

\newtheorem {lemma}{Lemma}

\newtheorem {proposition}{Proposition}

\input {tcilatex}

\begin{document}
\title[Hyperbolicity and shadowing lemma]{Characterization of hyperbolicity
and generalized shadowing lemma}
\author{Davor Dragi\v{c}evi\'{c} and Sini\v{s}a Slijep\v{c}evi\'{c}}
\address{Department of Mathematics, University of Zagreb, Bijenicka 30,
Zagreb, Croatia}
\email{ddragicevic@math.uniri.hr, slijepce@math.hr}
\date{June 17, 2011}
\subjclass[2000]{ 37D20, 37D25}
\keywords{Hyperbolicity, Lyapunov exponents, shadowing,\ Anosov maps}

\begin{abstract}
J. Mather characterized uniform hyperbolicity of a discrete dynamical system
as equivalent to invertibility of an operator on the set of all sequences
bounded in norm in the tangent bundle of an orbit. We develop a similar
characterization of nonuniform hyperbolicity and show that it is equivalent
to invertibility of the same operator on a larger, Fr\'{e}chet space. We
apply it to obtain a condition for a diffeomorphism on the boundary of the
set of Anosov diffeomorphisms to be nonuniformly hyperbolic. Finally we
generalise the Shadowing lemma in the same context.
\end{abstract}

\maketitle

\section{Introduction}

Assume $f$ is a $C^{1}$diffeomorphism on a finite dimensional Riemannian
smooth manifold $M$. Recall that a compact $f$-invariant set $\Lambda
\subseteq M$ is uniformly hyperbolic if for each $x\in \Lambda $ there exist
a decomposition $T_{x}M=E^{s}(x)\oplus E^{u}(x)$ and constants $c>0$ and $%
0<\lambda <1$, such that for each $x\in \Lambda $, $%
Df(x)E^{s}(x)=E^{s}(f(x)) $ and $Df(x)E^{u}(x)=E^{u}(f(x))$,%
\begin{eqnarray*}
|Df^{k}(x)\eta | &\leq &c\lambda ^{k}|\eta |\text{ whenever }\eta \in
E^{s}(x)\text{ and }k>0,\text{and} \\
|Df^{-k}(x)\eta | &\leq &c\lambda ^{k}|\eta |\text{ whenever }\eta \in
E^{u}(x)\text{ and }k>0.
\end{eqnarray*}

Recall that a $f$-invariant set $\Lambda $ is (nonuniformly)\ hyperbolic, if
each point of $\Lambda $ has non-zero Lyapunov exponents. An ergodic measure 
$\mu $ is hyperbolic, if $\mu $-a.e. point has non-zero Lyapunov exponents (%
\cite{Barreira:01}, \cite{Barreira:02}). Definitions are recalled in more
detail in Section 3.

We study here different characterizations of uniform and nonuniform
hyperbolicity, inspired by a characterization of J. Mather (\cite{Lanford90}%
, \cite{Mather68}), closely related to a concept referred to in the physics
literature as Thouless formula (\cite{Thouless:72}). We first introduce the
notation. We will always assume $M={{{{\mathbb{R}}}}}^{d}/{{{{\mathbb{Z}}}}}%
^{d}$ is a $d$ dimensional torus with the canonical Riemannian metric, for
simplicity of notation and clarity of arguments. All results are valid for
arbitrary finite dimensional smooth Riemannian manifolds, and can be easily
generalized by choosing appropriate local charts. The set $M^{{{{{\mathbb{Z}}%
}}}}$ will be the set of all sequences $\boldsymbol{x}{{{\mathbf{=(}}}}%
x_{k})_{k\in {{{{\mathbb{Z}}}}}}$, $x_{k}\in M$. A tangent space $T_{x}M$
will be naturally identified with ${{{{\mathbb{R}}}}}^{d}$, and the space $({%
{{{\mathbb{R}}}}}^{d})^{{{{{\mathbb{Z}}}}}}$ contains the subset of all
tangent orbits $\eta _{n+1}=Df(x_{n})\eta _{n}$.

Assume $X_{\infty }=l_{\infty }({{{{\mathbb{R}}}}}^{d})$ is the Banach space
of sequences ${{{\mathbf{{\eta }=(}}}}\eta _{k})_{k\in {{{{\mathbb{Z}}}}}}$, 
$\eta _{k}\in {{{{\mathbb{R}}}}}^{d}$ satisfying $\sup_{k\in {{{{\mathbb{Z}}}%
}}}|\eta _{k}|<\infty $, where $|\cdot |$ denotes the Euclidean norm on ${{{{%
\mathbb{R}}}}}^{d}$, and $||.||_{\infty }$ is the $\sup $ norm on $X_{\infty
}$. For any ${{{\mathbf{{x}\in }}}}M{{{\mathbf{^{{{{{\mathbb{Z}}}}}}}}}}$
(not necessarily an orbit) we define the operator $\Gamma _{{{{\mathbf{x}}}}%
}\colon X_{\infty }\rightarrow X_{\infty }$ by: 
\begin{equation}
(\Gamma _{{{{\mathbf{x}}}}}{{{\mathbf{{\eta })}}}}_{k}{{{\mathbf{=}}}}\eta
_{k}-Df(x_{k-1})\eta _{k-1}{{{\text{\textbf{, }}}}}k{{{\mathbf{\in }{{{%
\mathbb{Z}}}}\text{.}}}}  \label{d:gamma}
\end{equation}

It is easy to check that $\Gamma _{{{{\mathbf{x}}}}}$ is a well defined,
bounded linear operator on $X_{\infty }$.

We will denote by $o(x)\in M{{{\mathbf{^{{{{{\mathbb{Z}}}}}}}}}}$ the orbit
of $x\in M$. The operator $\Gamma _{o(x)}$ measures how much $\mathbf{\eta }$
differs from a tangent orbit. Assume $\Lambda \in M$ is a closed invariant
set. J. Mather proved the following characterization of uniform
hyperbolicity:\ $\Lambda $ is uniformly hyperbolic if and only if for each $%
x\in \Lambda $, $\Gamma _{o(x{{{\mathbf{)}}}}}$ has a continuous inverse
with the norm uniformly bounded in $\Lambda $. Mather's characterization is
actually in terms of operators on sections of the tangent bundle (\cite%
{Hasselblatt:02}, \cite{Mather68}). Here we use the equivalent setting from 
\cite{Lanford90}.

We consider this characterization of uniform hyperbolicity to be
particularly elegant and inspiring. For example, one can relatively easily
obtain all classical results of the theory of uniform hyperbolicity such as
shadowing, existence of stable and unstable manifolds and structural
stability, by an appropriate application of inverse or implicit function
theorems on Banach spaces, using the operator $\Gamma _{{{{\mathbf{x}}}}}$ (%
\cite{Lanford90}, \cite{Mather68}, \cite{MacKay95}).

We will need here a more general family of norms $||.||_{n}$, $1\leq n\leq
\infty $, on subsets of $({{{{\mathbb{R}}}}}^{d})^{{{{{\mathbb{Z}}}}}}$,
defined as 
\begin{equation}
||\mathbf{\eta }||_{n}=\sup_{k\in {{{{\mathbb{Z}}}}}}\func{exp}(-|k|/n)|\eta
_{k}|{{{\text{.}}}}  \label{d:normn}
\end{equation}%
Let $X_{n}$ be the set of all $\mathbf{\eta }$\textbf{\ }$\in ({{{{\mathbb{R}%
}}}}^{d})^{{{{{\mathbb{Z}}}}}}$ such that $||{{{\mathbf{\eta }}}}%
||_{n}<\infty $. We will also continue to write $X_{\infty }$ instead of the
usual notation $l_{\infty }({{{{\mathbb{R}}}}}^{d})$. As $M$ is compact, it
is easy to check that $\Gamma _{{{{\mathbf{x}}}}}$ is a well-defined,
bounded linear operator on $X_{n}$ for any positive integer $n$ and any $%
\mathbf{x}{{\mathbf{\in }}}M{{{\mathbf{^{{{{{\mathbb{Z}}}}}}}}}}$.

We will first somewhat extend the Mather's characterization of uniform
hyperbolicity to norms $||\boldsymbol{.}||_{n}$ in order to put the later
results in the right context.

\begin{theorem}
\label{t:uh}Assume $f$ is a diffeomorphism on $M$ and $\Lambda \subseteq M$
is a closed invariant set. Then $\Lambda $ is uniformly hyperbolic if and
only if for some $n_{0}$, $1\leq n_{0}\leq \infty $ and for each $x\in
\Lambda $, $\Gamma _{o(x{{{\mathbf{)}}}}}$ has a continuous inverse in $%
X_{n_{0}}$ such that $||\Gamma _{o(x{{{\mathbf{)}}}}}^{-1}||_{n_{0}}$ is
bounded uniformly in $\Lambda $.

Furthermore, if this holds for some $n_{0}$, then it holds for all\textbf{\ }%
sufficiently large $n$.
\end{theorem}

This is discussed and proved in Section 2.

Now let ${{{{\mathcal{N}}}}}=\bigcap_{n=1}^{\infty }X_{n}$ be the Fr\'{e}%
chet space of sequences in $({{{{\mathbb{R}}}}}^{d})^{{{{{\mathbb{Z}}}}}}$
growing sub-expontentially in norm. We then show that we can use the same
language and characterize nonuniform hyperbolicity, and prove the following
in Section 3:

\begin{theorem}
\label{t:nuh}Assume $f$ is a diffeomorphism on $M$, and $\mu $ an ergodic
measure. Then $\mu $ is (nonuniformly) hyperbolic if and only if for $\mu $%
-a.e. $x\in M$, $\Gamma _{o(x{{{\mathbf{)}}}}}$ has a continuous inverse in $%
{{{{\mathcal{N}}}}}$.
\end{theorem}

We discuss in the same section in some detail the structure of the space ${{{%
{\mathcal{N}}}}}$ with respect to the inverse of $\Gamma _{o(x{{{\mathbf{)}}}%
}}$, and develop further characterizations of nonuniform hyperbolicity more
suitable to applications.

As an application, we investigate in Section 4 limits of sequences of Anosov
diffeomorphisms. We show that if uniform hyperbolicity of Anosov
diffeomorphisms is destroyed on a set $A$ of measure 0 in such a way that
the speed of divergence of uniform bounds is not exponentially fast with
respect to the size of neighborhoods of $A$, then the limit is a
nonuniformly hyperbolic map. This application suggests that our
characterization may provide new insights, as the proof is analytical
(including relatively strong tools such as the Open mapping theorem for Fr%
\'{e}chet spaces) and ergodic-theoretical, but requires no geometric
information which is typically required in invariant cone and similar
techniques.

One of the key topological properties of uniformly and nonuniformly
hyperbolic sets is that one can prove a version of shadowing lemma, and as a
corollary that there are infinitely many periodic orbits. We can say that an
invariant set or measure $\mu $ is \textit{shadowable}, if such a shadowing
lemma holds (a more precise definition is in Section 5).

We finally show that something less than hyperbolicity is required to obtain
the shadowing property, and can be expressed as follows:

\begin{theorem}
\label{t:shadowing}Let $f$ be a $C^{1+\alpha }$ diffeomorphism on $M$, and $%
\Lambda $ an invariant set. Assume that there is some $n$, $1\leq n\leq
\infty $, such that for all $x\in \Lambda $, $\Gamma _{o(x{{{\mathbf{)}}}}}$
has a continuous inverse in $X_{n}$. Then $\Lambda $ is shadowable.
\end{theorem}

We prove it in Section 5, and also explain in which sense it generalizes the
Shadowing lemmas in uniformly and nonuniformly hyperbolic cases.

We now suggest various applications of that result. First, the shadowing of
the type described in Theorem\ \ref{t:shadowing} may hold in many dynamical
systems, such as twist and symplectic maps and more general Hamiltonian
systems. They often have rich regions of hyperbolic-like behavior with
shadowing phenomena, but nonuniform hyperbolicity on a set of positive
measure has not been proved and remains in many cases a conjecture.
Furthermore, it can be used in analysis of infinitely dimensional dynamical
systems, where it is difficult to use classical definitions of uniform and
nonuniform hyperbolicity. Finally, one may try to apply it for further
investigation of relationship of shadowing and stochastic stability (\cite%
{Bonatti:05}, Problem D10).

\section{Uniform hyperbolicity}

In this section $f$ is a $C^{1}$ diffeomorphism on $M$. We introduce a
family of norms on $({{{{\mathbb{R}}}}}^{d})^{{{{{\mathbb{Z}}}}}}$ more
general than the norms $X_{n}$, $1\leq n\leq \infty $. Let $\boldsymbol{w}%
=(w_{k})_{k\in {{{{\mathbb{Z}}}}}}$ be a sequence of real numbers such that $%
w_{k}>0$ and such that for some $b>0$, 
\begin{equation}
\sup_{k\in {{{\mathbf{Z}}}}}\frac{w_{k-1}}{w_{k}}<b,\ \sup_{k\in {{{\mathbf{Z%
}}}}}\frac{w_{k}}{w_{k-1}}<b.  \label{r:n0}
\end{equation}

Assume $N$ is a norm defined as 
\begin{equation*}
N({{{\mathbf{\eta }}}})=\sup_{k\in \boldsymbol{Z}}w_{k}|\eta _{k}|,
\end{equation*}%
and let $X_{\boldsymbol{w}}$ be the set of all ${{{\mathbf{\eta }}}}\in ({{{{%
\mathbb{R}}}}}^{d})^{{{{{\mathbb{Z}}}}}}$ such that $N({{{\mathbf{\eta }}}}%
)<\infty $. If $\Upsilon $ is a linear operator on $X_{\boldsymbol{w}}$
which can be represented as a matrix $\Upsilon =(\Upsilon _{i,j})$, $i,j\in {%
{{{\mathbb{Z}}}}}$, where $\Upsilon _{i,j}$ is a linear operator on ${{{{%
\mathbb{R}}}}}^{d}$, then the norm $N(\Upsilon )$ can be bounded with 
\begin{equation}
N(\Upsilon )\leq \sup_{i\in \boldsymbol{Z}}w_{i}\sum_{j\in \boldsymbol{Z}}%
\frac{1}{w_{j}}|\Upsilon _{i,j}|{{{\text{.}}}}  \label{r:norm}
\end{equation}%
(The results such as (\ref{r:norm}) on operators with representations as
infinite matrices are summarized in Section 6:\ Appendix). Recall the
definition of $\Gamma _{o(x)}$ in Section 1, and that we assume that $M={{{{%
\mathbb{R}}}}}^{d}/{{{{\mathbb{Z}}}}}^{d}$, hence compact. It is easy to
check that (\ref{r:n0}), (\ref{r:norm}) imply that if $x\in M$, then $\Gamma
_{o(x)}$ is a bounded linear operator on $X_{\boldsymbol{w}}$.

First note that for a uniformly hyperbolic set $\Lambda $, the angle between 
$E^{s}(x)$ and $E^{u}(x)\ $is uniformly bounded away from zero (e.g. \cite%
{Katok:95}, Corollary 6.4.5).\ Therefore there exists a constant $a>0$ such
that, if $\eta =\eta ^{s}+\eta ^{u}$ is the hyperbolic splitting of $\eta
\in T_{x}M$, then 
\begin{equation}
|\eta ^{s}|\leq a|\eta |,|\eta ^{u}|\leq a|\eta |.  \label{r:cangle}
\end{equation}
Similarly, if $\mathbf{\eta }\in X_{\boldsymbol{w}}$ and $\mathbf{\eta =\eta 
}^{s}+\mathbf{\eta }^{u}$ is the hyperbolic splitting in each component
along some sequence in $\Lambda $, then%
\begin{equation}
N(\mathbf{\eta }^{s})\leq aN(\mathbf{\eta }),\text{ }N(\mathbf{\eta }%
^{u})\leq aN(\mathbf{\eta }).  \label{r:nice}
\end{equation}

\begin{proposition}
\label{p:uh1}Let $\Lambda \subset M$ be a closed invariant uniformly
hyperbolic set. There exists $\delta >0$ such that for any norm $N$
satisfying (\ref{r:n0}) with $1\leqslant b\leqslant 1+\delta $, the
following holds: for all $x\in M$, $\Gamma _{o(x)}$ is invertible in $X_{%
\boldsymbol{w}}$ and $N(\Gamma _{o(x)}^{-1})$ is uniformly bounded in $%
\Lambda $.
\end{proposition}

\begin{proof}
Let $c>0$, $\lambda <1$ be as in the definition of uniformly hyperbolic set.
Take any $\delta >0$ such that $\lambda (1+\delta )<1$ and choose any $x\in
\Lambda $. Since (\ref{r:n0}) holds and $b\leqslant 1+\delta $ as assumed,
for any $k\in {{{{\mathbb{Z}}}}}$ and $l\geqslant 0$ 
\begin{equation}
\frac{w_{k}}{w_{k-l}}\leqslant (1+\delta )^{l},\ \frac{w_{k}}{w_{k+l}}%
\leqslant (1+\delta )^{l}{{{\text{.}}}}  \label{r:h2}
\end{equation}%
\newline
\textit{Injectivity of }$\Gamma _{o(x)}$\textit{\ on }$X_{\boldsymbol{w}}$%
\textit{. }Assume that $\Gamma _{o(x)}\boldsymbol{\eta }=0$ for some $%
\boldsymbol{\eta }\in X_{\boldsymbol{w}}$. For any $k\in {{{{\mathbb{Z}}}}}$
we can write $\eta _{k}=\eta _{k}^{s}+\eta _{k}^{u}$, where $\eta
_{k}^{s}\in E^{s}(x_{k})$ and $\eta _{k}^{u}\in E^{u}(x_{k})$. Definition of
uniform hyperbolicity then implies that for any $k\in {{{{\mathbb{Z}}}}}$, $%
\eta _{k}^{s}=Df(x_{k-1})\eta _{k-1}^{s}$ and $\eta _{k}^{u}=Df(x_{k-1})\eta
_{k-1}^{u}$. Assume now $\eta _{j}^{s}\neq 0$ for some integer $j$. Then for
all $l\geqslant 0$, 
\begin{equation}
|\eta _{j}^{s}|=|Df^{l}(x_{j-l})\eta _{j-l}^{s}|\leqslant c\lambda ^{l}|\eta
_{j-l}^{s}|{{{\text{.}}}}  \label{r:h3}
\end{equation}

Combining inequalities (\ref{r:h2}) and (\ref{r:h3}) we conclude that for
all $l\geqslant 0$ 
\begin{equation}
w_{j-l}|\eta _{j-l}^{s}|\geqslant \frac{1}{c(\lambda (1+\delta ))^{l}}%
w_{j}|\eta _{j}^{s}|.  \label{r:h35}
\end{equation}

As $\lambda (1+\delta )<1$, we get from (\ref{r:h35}) that $w_{j-l}|\eta
_{j-l}^{s}|$ diverges to $\infty $ as $l\rightarrow \infty $, hence ${{{%
\mathbf{\eta }}}}^{s}\notin X_{\boldsymbol{w}}$, which is in contradiction
with (\ref{r:nice}). Similarly we get that $\eta _{j}^{u}=0$ for all $j\in {{%
{{\mathbb{Z}}}}}$, therefore ${{{\mathbf{\eta }}}}=0$.

\textit{Surjectivity of }$\Gamma _{o(x)}$\textit{\ on }$X_{\boldsymbol{w}}$%
\textit{. }Let $\mathbf{\eta }\in X_{\boldsymbol{w}}$. For each integer $k$
we define 
\begin{eqnarray*}
\xi _{k}^{s} &=&\sum_{l\geq 0}Df^{l}(x_{k-l})\eta _{k-l}^{s}, \\
\xi _{k}^{u} &=&-\sum_{l\geq 1}Df^{-l}(x_{k+l})\eta _{k+l}^{u}.
\end{eqnarray*}

Now using (\ref{r:h2}) and the definition of $N$ we get 
\begin{eqnarray*}
w_{k}|\xi _{k}^{s}| &\leqslant &\sum_{l\geq 0}w_{k}c\lambda ^{l}|\eta
_{k-l}^{s}|=\sum_{l\geq 0}\frac{w_{k}}{w_{k-l}}c\lambda ^{l}w_{k-l}|\eta
_{k-l}^{s}| \\
&\leqslant &\sum_{l\geqslant 0}c(\lambda (1+\delta ))^{l}w_{k-l}|\eta
_{k-l}^{s}|\leqslant \frac{c}{1-\lambda (1+\delta )}N(\boldsymbol{\eta }%
^{s}).
\end{eqnarray*}%
We deduce the series in the definition of $\xi _{k}^{s}$ is absolutely
convergent, that $\boldsymbol{\xi }^{s}:=(\xi _{k}^{s})\in X_{\boldsymbol{w}%
} $ and that 
\begin{equation}
N(\boldsymbol{\xi }^{s})\leqslant \frac{c}{1-\lambda (1+\delta )}N(%
\boldsymbol{\eta }^{s}).  \label{r:h4}
\end{equation}%
Similarly one can show that $\boldsymbol{\xi }^{u}:=(\xi _{k}^{u})_{k}\in X_{%
\boldsymbol{w}}$ and 
\begin{equation}
N(\boldsymbol{\xi }^{u})\leqslant \frac{c\lambda (1+\delta )}{1-\lambda
(1+\delta )}N(\boldsymbol{\eta }^{u}).  \label{r:h5}
\end{equation}%
The triangle inequality, (\ref{r:nice}), (\ref{r:h4}) and (\ref{r:h5}) now
yield 
\begin{equation}
N(\boldsymbol{\xi })\leqslant ac\frac{1+\lambda (1+\delta )}{1-\lambda
(1+\delta )}N(\boldsymbol{\eta }){{{\text{.}}}}  \label{r:h6}
\end{equation}

It is easy to verify that $\Gamma _{o(x)}\boldsymbol{\xi }=\boldsymbol{\eta }
$. Also the constant on the right-hand side of (\ref{r:h6}) is the uniform
bound on $N(\Gamma _{o(x)}^{-1})$.
\end{proof}

The following is a result from \cite{Mather68} adapted to our setting as in 
\cite{Lanford90}.

\begin{proposition}
\label{p:uh2}Assume $\Lambda \subseteq M$ is a closed, invariant set, such
that for each $x\in M$, $\Gamma _{o(x)}$ is invertible in $X_{\infty }$ and $%
||\Gamma _{o(x)}^{-1}||_{\infty }$ is bounded uniformly for $x\in \Lambda $.
Then $\Lambda $ is uniformly hyperbolic.
\end{proposition}

\begin{proof}
Assume $c_{1}$ is\ the uniform bound on $||\Gamma _{o(x)}^{-1}||_{\infty }$
and choose $x_{0}\in \Lambda $. Choose any $\theta \in {{{{\mathbb{R}}}}}%
^{d}=T_{x_{0}}M$. Let $\eta _{0}=\theta $ and $\eta _{k}=0$ if $k\neq 0$.
Clearly $\boldsymbol{\eta }=(\eta _{k})\in X_{\infty }$. Since $\Gamma
_{o(x_{0})}$ is invertible, there exists $\boldsymbol{\xi }\in X_{\infty }$
such that $\Gamma _{o(x_{0})}\boldsymbol{\xi }=\boldsymbol{\eta }$. This can
be written as 
\begin{eqnarray}
\xi _{k} &=&Df(x_{k-1})\xi _{k-1}{{{\text{, }}}}k\neq 0{{{\text{,}}}}
\label{r:hyp1} \\
\xi _{0} &=&Df(x_{-1})\xi _{-1}+\theta {{{\text{.}}}}  \notag
\end{eqnarray}

We will now show that $\theta =\theta ^{s}+\theta ^{u}$, where $\theta
^{s}=\xi _{0}$ and $\theta ^{u}=-Df(x_{-1})\xi _{-1}$ is the hyperbolic
splitting. We define a family of operators $B(z)$, $z\geqslant 1$ on $%
X_{\infty }$ as 
\begin{equation*}
(B(z)\boldsymbol{\nu })_{k}=\QATOPD\{ . {\nu _{k}-(1/z)A_{k-1}\nu _{k-1}%
\text{, }k\leq 0}{\nu _{k}-zA_{k-1}\nu _{k-1}\text{, }k\geq 1.}.
\end{equation*}%
Note that $B(1)=\Gamma _{o(x_{0})}$. Since $Df$ is continuous and $M$ is
compact, let $c_{2}$ be the maximum of $|Df(x)|$ over $M$. Inequalities $%
||B(z)-B(1)||_{\infty }\leqslant c_{2}(z-1)$ and $||B(1)^{-1}||_{\infty
}\leqslant c_{1}$ imply that $B(z)$ is an invertible operator for $%
1\leqslant z<1+1/(c_{1}c_{2})$ and 
\begin{equation}
||B(z)^{-1}||_{\infty }\leqslant \frac{1}{c_{1}^{-1}-c_{2}(z-1)},
\label{r:13b}
\end{equation}%
as $B(z)^{-1}=B(1)^{-1}\sum_{k=0}^{\infty }((B(1)-B(z))B(1)^{-1})^{k}$. Now
choose any $\lambda $ such that $1<1/\lambda <1+1/(c_{1}c_{2})$, and find a
(unique)\ $\boldsymbol{\xi }^{\ast }$ such that $B(\lambda )\boldsymbol{\xi }%
^{\ast }=\boldsymbol{\eta }$. If $c_{3}$ is the right-hand side of (\ref%
{r:13b}) when $z=\lambda $, then 
\begin{equation*}
||\boldsymbol{\xi }^{\ast }||_{\infty }\leq ||B(\lambda )^{-1}||_{\infty }||%
\boldsymbol{\eta }||_{\infty }\leq c_{3}||\boldsymbol{\eta }||_{\infty
}=c_{3}|\theta |\text{.}
\end{equation*}

As $\boldsymbol{\xi }^{\ast \ast }$ defined as $\xi _{k}^{\ast \ast
}=\lambda ^{|k|}\xi _{k}^{\ast }$ clearly belongs to $X_{\infty }$ and by
definition of $B(\lambda )$, $\Gamma _{o(x_{0})}\boldsymbol{\xi }^{\ast \ast
}=\boldsymbol{\eta }$, it must be $\boldsymbol{\xi }^{\ast \ast }=%
\boldsymbol{\xi }$, thus 
\begin{equation}
|\xi _{k}|\leq c_{3}\lambda ^{|k|}|\theta |.  \label{r:13c}
\end{equation}
Because of (\ref{r:hyp1}) and the definition of $\theta ^{s}$, $\theta ^{u}$%
, for all $k\geq 1$ it is 
\begin{equation}
\xi _{k}=Df^{k}(x_{0})\theta ^{s},\xi _{-k}=Df^{-k}(x_{0})\theta ^{u}.
\label{r:13d}
\end{equation}

Denote by $E^{s}(x_{0})$, $E^{u}(x_{0})$ the sets of all $\theta ^{s}$, $%
\theta ^{u}$ constructed as above. The sets $E^{s}(x_{0})$, $E^{u}(x_{0})$
are linear subspaces, as images of linear maps 
\begin{equation}
P^{s}=P^{0}\circ \Gamma _{o(x_{0})}^{-1}\circ J\text{, }P^{u}=I-P^{s},
\label{r:13e}
\end{equation}
where $P^{u},P^{s}:{{{{\mathbb{R}}}}}^{d}\rightarrow {{{{\mathbb{R}}}}}^{d}$%
, $J:{{{{\mathbb{R}}}}}^{d}\rightarrow X_{\infty }$ is the inclusion $\theta
\mapsto (...,0,0,\theta ,0,..)$ and $P^{0}$ projection to the $0$%
-coordinate. The construction implies that $P^{s}$, $P^{u}$ are identities
on $E^{s}(x_{0})$, $E^{u}(x_{0})$ respectively, therefore and because of $%
I=P^{u}+P^{s}$ we get $T_{x_{0}}M=E^{s}(x_{0})\oplus E^{u}(x_{0})$.
Invariance of $E^{s}$, $E^{u}$ with respect of $Df$ follows from (\ref{r:13e}%
), and uniformly hyperbolic inequalities from (\ref{r:13c}), (\ref{r:13d})
by inserting $\theta =\theta ^{s}$, $\theta =\theta ^{u}$.
\end{proof}

The following Lemma gives further insight on the relationship of the graded
norms $||.||_{n}$ and invertibility of $\Gamma _{\boldsymbol{x}}$. Let $S$
denote the shift $(S\mathbf{\eta })_{k}=\eta _{k-1}$. Note that for any norm 
$N$ with the property (\ref{r:n0}), $S$ is a bounded linear operator with a
continuous inverse on $X_{\boldsymbol{w}}$. Recall the definition of the
norm and space $X_{n}$ in (\ref{d:normn}). In the following we use results
from the Appendix regarding matrix representation of operators on $X_{n}$, $%
X_{\infty }$. In particular, if $\Gamma _{\mathbf{x}}$ is invertible with a
continuous inverse, according to Lemma \ref{l:matrix} its inverse has a
matrix representation.

\begin{lemma}
\label{l:uh3}Assume $\mathbf{x}\in M^{{{{{\mathbb{Z}}}}}}$ (not necessarily
an orbit). Assume that for some constants $n\in {{{{\mathbb{N}}}}}$ and $%
c_{4}>0$, for all integers $k$, $\Gamma _{S^{k}\boldsymbol{x}}$ is
invertible in $X_{n}$ and $||\Gamma _{S^{k}\boldsymbol{x}}^{-1}||_{n}\leq
c_{4}$. Then $\Gamma _{\boldsymbol{x}}$ is invertible in $X_{\infty }$ and 
\begin{equation*}
||\Gamma _{\boldsymbol{x}}^{-1}||_{\infty }\leq 2c_{4}d\sqrt{d}/(1-\func{exp}%
(-1/n)){{{\text{.}}}}
\end{equation*}
\end{lemma}

\begin{proof}
Let $\Upsilon $ be the inverse of $\Gamma _{\boldsymbol{x}}$ with the matrix
representation $(\Upsilon _{i,j}),i,j\in \boldsymbol{Z}$, and $\Upsilon
^{(k)}$ the inverse of $\Gamma _{S^{k}\boldsymbol{x}}$, all in $X_{n}$. As $%
\Gamma _{S^{k}\boldsymbol{x}}=S^{-k}\Gamma _{\boldsymbol{x}}S^{k}$, and $S$
is automorphism of $X_{n}$, we deduce that $\Upsilon ^{(k)}=S^{-k}\Upsilon
S^{k}$. Using that and (\ref{r:44}), we obtain that for all $k\in 
\boldsymbol{Z}$, 
\begin{equation*}
\sup_{i}\sum_{j}\func{exp}((|j|-|i|)/n)|\Upsilon _{i+k,j+k}|/(d\sqrt{d})\leq
||\Upsilon ^{(k)}||_{n}\leq c_{4}{{{\text{.}}}}
\end{equation*}

By choosing $i=0$, $j=j_{0}-i_{0}$, $k=i_{0}$, we deduce from it that for
any $i_{0},j_{0}$, $|\Upsilon _{i_{0},j_{0}}|\leq c_{4}d\sqrt{d}\func{exp}%
(-|j_{0}-i_{0}|/n)$. We now get that 
\begin{equation*}
||\Upsilon ||_{\infty }\leq \sup_{i}\sum_{j}|\Upsilon _{i,j}|\leq 2c_{4}d%
\sqrt{d}/(1-\func{exp}(-1/n)){{{\text{.}}}}
\end{equation*}
\end{proof}

Theorem \ref{t:uh} follows directly from Propositions \ref{p:uh1}, \ref%
{p:uh2} and Lemma \ref{l:uh3}.

\section{Nonuniform hyperbolicity}

Prior to focusing on nonuniform hyperbolicity we will introduce a Fr\'{e}%
chet space ${{{{\mathcal{N}}}}}$ which will be the key in the
characterization of nonuniform hyperbolicity. Define ${{{{\mathcal{N}}}}}$
as the set of all sequences ${{\mathbf{\eta }}}{{{\mathbf{=}}}}(\eta
_{k})_{k\in {{{{\mathbb{Z}}}}}}$, $\eta _{k}\in {{{{\mathbb{R}}}}}^{d}$,
satisfying 
\begin{equation*}
\limsup_{k}\frac{1}{|k|}\log |\eta _{k}|\leq 0.
\end{equation*}%
We can naturally equip ${{{{\mathcal{N}}}}}$ with the structure of a vector
space. Also ${{{{\mathcal{N}}}}}=\bigcap_{n=1}^{\infty }X_{n}$, where $X_{n}$
is as defined in the introduction, and the norms $||.||_{n}$ are graded in
the sense that $||{{\mathbf{\eta }}}||_{1}\leq ||{{\mathbf{\eta }}}%
||_{2}\leq ||{{\mathbf{\eta }}}||_{3}\leq \ldots $. The topology on ${{{{%
\mathcal{N}}}}}$ is the topology generated by the family of norms $||.||_{n}$%
, $n\in \boldsymbol{N}$. Recall that a complete, locally convex topological
vector space with a topology induced by a translation-invariant metric is
called a Fr\'{e}chet space (\cite{Rudin:91}).

\begin{proposition}
${{{{\mathcal{N}}}}}$ with the topology induced by the family of norms $%
||\cdot ||_{n}$, $n\in {{{{\mathbb{N}}}}}$ forms a Fr\'{e}chet space.
\end{proposition}

\begin{proof}
A countable family of norms always induces a locally convex topology which
can be generated by a translation-invariant metric (\cite{Rudin:91}), so we
only need to prove completeness of ${{{{\mathcal{N}}}}}$. Assume that $({{%
\mathbf{\eta }}}^{k})_{k\in {{{{\mathbb{N}}}}}}$ is a Cauchy sequence in ${{{%
{\mathcal{N}}}}}$, so by definition it is Cauchy in $X_{n}$ for all $n$,
hence convergent to some $\boldsymbol{\xi }^{n}$. As the norms $||.||_{n}$
are graded, $\boldsymbol{\xi }^{n}=:\boldsymbol{\xi }$ is independent of $n$%
, so ${{\mathbf{\eta }}}^{k}$ converge to $\boldsymbol{\xi }$ in ${{{{%
\mathcal{N}}}}}$.
\end{proof}

We noted in the introduction that $\Gamma _{o(x)}$ is a well defined
continuous linear operator on $X_{n}$ for any positive integer $n$ and any ${%
{\mathbf{x}}}\in M^{{{{{\mathbb{Z}}}}}}$, so $\Gamma _{o(x)}$ is a well
defined continuous linear operator on ${{{{\mathcal{N}}}}}$. It is also easy
to check that $\Gamma _{o(x)}:X_{n}\rightarrow X_{m}$ is a well defined and
continuous for any positive integers $n\geq m$.

Recall that an ergodic $f$-invariant Borel probability measure $\mu $ on $M$
is hyperbolic if none of the Lyapunov exponents are zero. We now extend
Theorem \ref{t:nuh} and introduce several related characterizations of
nonuniform hyperbolicity.

\begin{theorem}
\label{t:mainhyp}Say $f$ is a diffeomorphism on $M$, and $\mu $ an ergodic $%
f $-invariant Borel probability measure. Then the following is equivalent:

(i) The measure $\mu $ is hyperbolic.

(ii) For $\mu $-a.e. $x\in M$, $\Gamma _{o(x)}$ is bijective on ${{{{%
\mathcal{N}}}}}$.

(iii) For $\mu $-a.e. $x\in M$, $\Gamma _{o(x)}$ has a continuous inverse on 
${{{{\mathcal{N}}}}}$.

(iv) There exists a positive integer $n_{0}$ such that for any $n>n_{0}$, $%
n>m$, and for $\mu $-a.e. $x\in M$ there exists a continuous operator $%
\Upsilon :X_{n}\rightarrow X_{m}$ which is a left and right inverse of $%
\Gamma _{o(x)}$ in appropriate spaces.

(v) There exists increasing sequences of integers $n_{i}>m_{i}$ both
converging to $\infty $ so that for $\mu $-a.e. $x\in M$ there exist
continuous operators $\Upsilon _{i}:X_{n_{i}}\rightarrow X_{m_{i}}$ which
are left and right inverses of $\Gamma _{o(x)}$ in appropriate spaces.
\end{theorem}

In the next section we further adapt this characterization to applications
(see Corollary \ref{c:nicechar}).

Recall that a point $x\in M$ is regularly hyperbolic, if there exists $%
0<\lambda <1$ and for each sufficiently small $\varepsilon >0$ a constant $%
c(x,\varepsilon )$ so that the following holds: there exists a decomposition 
$T_{x_{k}}M=E^{s}(x_{k})\oplus E^{u}(x_{k})$ along the orbit $x_{k}=f^{k}(x)$
with the invariance property $Df(x_{k})E^{u}(x_{k})=E^{u}(x_{k+1})$, $%
Df(x_{k})E^{s}(x_{k})=E^{s}(x_{k+1})$ and such that%
\begin{eqnarray*}
|Df^{j}(x_{k})\eta | &\leq &c(x,\varepsilon )\lambda ^{j}e^{\varepsilon
|k|}|\eta |\text{ whenever }\eta \in E^{s}(x_{k})\text{ and }j>0\text{, and}
\\
|Df^{-j}(x_{k})\eta | &\leq &c(x,\varepsilon )\lambda ^{j}e^{\varepsilon
|k|}|\eta |\text{ whenever }\eta \in E^{u}(x_{k})\text{ and }j>0\text{.}
\end{eqnarray*}

In addition, the angle between $E^{u}(x_{k})$ and $E^{s}(x_{k})$ is bounded,
so that if $\eta =\eta ^{s}+\eta ^{u}$ is the hyperbolic splitting at $%
T_{x_{k}}M$, then%
\begin{equation*}
|\eta ^{s}|\leq c(x,\varepsilon )e^{\varepsilon |k|}|\eta |\text{, }|\eta
^{u}|\leq c(x,\varepsilon )e^{\varepsilon |k|}|\eta |\text{.}
\end{equation*}

Pesin proved that if an ergodic $f$-invariant measure $\mu $ is hyperbolic,
then\ $\mu $-a.e. $x\in M$ is regularly hyperbolic and the function $%
c(x,\varepsilon )$ is Borel measurable for all small enough $\varepsilon >0$
(see e.g. \cite{Barreira:02}, Theorem 2.1.3).

We say that $x\in M$ is \textit{u-hyperbolic}, if there exist two subspaces $%
E^{u}(x)$ and $E^{s}(x)$ spanning $T_{x}M$ such that for all $\eta ^{u}\in
E^{u}\setminus \{0\}$, $\eta ^{s}\in E^{s}\setminus \{0\}$, 
\begin{eqnarray*}
\limsup_{k\rightarrow \infty }\frac{1}{k}\log |Df^{-k}(x)\eta ^{s}| &>&0, \\
\limsup_{k\rightarrow \infty }\frac{1}{k}\log |Df^{k}(x)\eta ^{u}| &>&0.
\end{eqnarray*}

We say that an ergodic $f$-invariant Borel probability measure $\mu $ on $M$
is \textit{u-hyperbolic} if $\mu -$a.e point is u-hyperbolic.

We prove Theorem \ref{t:mainhyp} in a series of Lemmas.

\begin{lemma}
\label{l:nh1} Every $u$-hyperbolic measure is hyperbolic.
\end{lemma}

\begin{proof}
Suppose that $\mu $ is an $u$-hyperbolic measure which is not hyperbolic. By
Oseledec theorem there exists a Borel measurable set $B$ with $\mu (B)=1$
such that for all $x\in B$ there exists a splitting $T_{x}M=E^{-}(x)\oplus
E^{0}(x)\oplus E^{+}(x)$ of subspaces with positive, zero and negative
Lyapunov exponents respectively, and $\func{dim}E^{0}(x)\geq 1$.

Let $B^{\prime }\subseteq B$ be a set of full measure of u-hyperbolic
points, and let $E^{u}(x)$, $E^{s}(x)$ be two subspaces spanning $T_{x}M$
from the definition of u-hyperbolicity. Now by comparing the dimensions and
discussing intersections of $E^{-}(x)$,$E^{0}(x),E^{+}(x)$ and $%
E^{u}(x),E^{s}(x)$ it is easy to deduce contradiction.
\end{proof}

\begin{lemma}
\label{l:previous}Assume that $x\in M$ and that $\Gamma _{o(x)}$ is an
injective operator on ${{{{\mathcal{N}}}}}$. If $\mathbf{\eta }=(\eta _{k})$
is a non-zero tangent orbit at $o(x)$, then either $\limsup_{k\rightarrow
\infty }(1/k)\log |\eta _{k}|>0$ or $\limsup_{k\rightarrow \infty }(1/k)\log
|\eta _{-k}|>0$.
\end{lemma}

\begin{proof}
Assume the contrary and find such non-zero ${{\mathbf{\eta }}}$ which is
then by assumption in ${{{{\mathcal{N}}}}}$. As ${{\mathbf{\eta }}}$ is a
tangent orbit, so is $\Gamma _{o(x)}{{\mathbf{\eta }}}=0$, which contradicts
the injectivity of $\Gamma _{o(x)}$.
\end{proof}

\begin{lemma}
\label{l:nh2} If $x\in M$ and $\Gamma _{o(x)}$ is bijective on ${{{{\mathcal{%
N}}}}}$, then $x$ is a u-hyperbolic point.
\end{lemma}

\begin{proof}
Suppose $\Gamma _{o(x)}$ is bijective on ${{{{\mathcal{N}}}}}$ and choose $%
\theta \in T_{x}M$. Define a ${{{\mathbf{\eta \in {{{{\mathcal{N}}}}}}}}}$
with $\eta _{0}=\theta $ and $\eta _{k}=0$ for $k$ non-zero. As $\Gamma
_{o(x)}$ is surjective, there exists ${{{\mathbf{{\mu }\in {{{{\mathcal{N}}}}%
}}}}}$ such that $\Gamma _{o(x)}{{{\mathbf{{\mu }={\eta }}}}}$. If we put $%
\theta ^{s}=\mu _{0}$ and $\theta ^{u}=-Df(f^{-1}(x))\mu _{-1}$, then $%
\theta =\theta ^{s}+\theta ^{u}$. Now let $\mu _{k}^{s}:=\mu _{k}$ for $%
k\geq 0$, and let 
\begin{equation}
\mu _{-k}^{s}=Df^{-k}(x)\theta ^{s}  \label{r:help0}
\end{equation}
for all $k\geqslant 1$. If $\theta ^{s}$ is non-zero then one easily checks
that ${{\mathbf{\mu }}}^{s}=(\mu _{k}^{s})$ is a non-zero tangent orbit. As
then $\Gamma _{o(x)}{{\mathbf{\mu }}}^{s}=0$ and by construction $%
\limsup_{k\rightarrow \infty }(1/k)\log |\eta _{k}^{s}|\leq 0$, injectivity
of $\Gamma _{o(x)}$, Lemma \ref{l:previous} and (\ref{r:help0}) imply that 
\begin{equation*}
\limsup_{k\rightarrow \infty }(1/k)\log |Df^{-k}(x)\theta ^{s}|>0.
\end{equation*}%
Analogously we show that $\limsup_{k\rightarrow \infty }(1/k)\log
|Df^{k}(x)\theta ^{u}|>0$ if $\theta ^{u}\neq 0$.

If $E^{s}(x_{0})$, $E^{u}(x_{0})$ are the sets of all $\theta ^{s}$, $\theta
^{u}$ constructed as above, we see that they are linear subspaces as images
of linear maps as constructed in the proof of Proposition \ref{p:uh2};\
spanning $T_{x}M$ also by construction.
\end{proof}

Suppose now that $x\in M$ is regularly hyperbolic. We now construct a left
and right inverse of $\Gamma _{o(x)}$ in appropriate spaces. Say $\mathbf{%
\mu }\in X_{n}$ for some positive integer $n$. As in the proof of
Proposition \ref{p:uh1}, we set%
\begin{eqnarray}
\eta _{k}^{s} &=&\sum_{l\geq 0}Df^{l}(x_{k-l})\mu _{k-l}^{s},  \label{r:one}
\\
\eta _{k}^{u} &=&-\sum_{l\geq 1}Df^{-l}(x_{k+l})\mu _{k+l}^{u},
\label{r:two}
\end{eqnarray}%
where $\mu _{k}=\mu _{k}^{s}+\mu _{k}^{u}$ is the hyperbolic splitting at $%
x_{k}=f^{k}(x)$.

We now use the definition of regular hyperbolicity and calculate for some
positive integer $m$:%
\begin{eqnarray}
e^{-|k|/m}|\eta _{k}^{s}| &\leq &e^{-|k|/m}\sum_{l\geq 0}c(x,\varepsilon
)\lambda ^{l}e^{\varepsilon |k-l|}|\mu _{k-l}^{s}|\leq  \notag \\
&\leq &e^{-|k|/m}\sum_{l\geq 0}c^{2}(x,\varepsilon )\lambda
^{l}e^{2\varepsilon |k-l|}e^{|k-l|/n}e^{-|k-l|/n}|\mu _{k-l}|\leq  \notag \\
&\leq &||\mu ||_{n}\cdot e^{(2\varepsilon +1/n-1/m)|k|}\cdot \sum_{l\geq
0}c^{2}(x,\varepsilon )e^{(2\varepsilon +1/n-\log (1/\lambda ))|l|}\text{.}
\label{r:bigsum}
\end{eqnarray}

If $n$ is large enough such that $1/n-\log (1/\lambda )<0$ and $n>m$,
clearly we can choose $\varepsilon >0$ small enough so that the expression
in (\ref{r:bigsum})\ multiplying $||\mu ||_{n}$ is convergent and bounded
uniformly in $k$ (that means, the bound depends only on $x,\varepsilon ,n,m$%
). We deduce that $\mathbf{\eta }^{s}\in X_{m}$, and also that the sum in (%
\ref{r:one}) is absolutely convergent. Analogously $\mathbf{\eta }^{u}\in
X_{m}$ and the sum in (\ref{r:two}) is absolutely convergent. Noting the
uniform bounds in (\ref{r:bigsum}) and its analogue for $\mathbf{\eta }^{u}$%
, we deduce that the operator $\Upsilon :\mathbf{\mu \mapsto \eta }$ is a
continuous operator $\Upsilon :X_{n}\rightarrow X_{m}$. As both $\Upsilon $
and $\Gamma _{o(x)}$ are operators with a matrix representation (see
Appendix for a discussion), it is easy to check that%
\begin{equation}
\Upsilon \circ \Gamma _{o(x)}=\Gamma _{o(x)}\circ \Upsilon =I
\label{r:inverse}
\end{equation}%
where $\Gamma _{o(x)}:X_{n}\rightarrow X_{n}$, respectively $\Gamma
_{o(x)}:X_{m}\rightarrow X_{m}$, and $I$ is the identity operator. We
summarize:

\begin{lemma}
\label{l:nm}Assume that $x\in M$ is regularly hyperbolic with a coefficient $%
0<\lambda <1$. Then for any $n>1/(\log (1/\lambda ))$ and $n>m$, where $n,m$
are positive integers, $\Gamma _{o(x)}$ has a left and right continuous
inverse $\Upsilon :X_{n}\rightarrow X_{m}$.
\end{lemma}

\begin{corollary}
\label{c:nh3} Assume that $x\in M$ is a regularly hyperbolic point. Then $%
\Gamma _{o(x)}$ has a continuous inverse on ${{{{\mathcal{N}}}}}$.
\end{corollary}

\begin{proof}
As $X_{n}$ is a decreasing family of sets, the operator $\Upsilon $
constructed in Lemma \ref{l:nm} does not depend on $n,m$ (this also follows
from its construction). We deduce that it is well-defined and continuous on $%
{{{{\mathcal{N}}}}}$. The relation (\ref{r:inverse}) restricted to ${{{{%
\mathcal{N}}}}}$ completes the proof.
\end{proof}

We now complete the proof of Theorem \ref{t:mainhyp}.

\begin{proof}
Lemma \ref{l:nm} and the Pesin theorem imply (i)$\Longrightarrow $(iv). The
implications (iv)$\Longrightarrow $(v) and (iii)$\Rightarrow $(ii) are
trivial. If (v) is true, as $X_{n}$ is a decreasing family of sets, it is
easy to deduce that $\Upsilon _{i}$ is independent of $i$. Therefore $%
\Upsilon :=\Upsilon _{i}$ is well defined and continuous on ${{{{\mathcal{N}}%
}}}$ and the inverse of $\Gamma _{o(x)}$, so (v)$\Longrightarrow $(iii).
Combining Lemmas \ref{l:nh1} and \ref{l:nh2} we get (ii)$\Longrightarrow $%
(i).
\end{proof}

\section{Nonuniform hyperbolicity on the boundary of Anosov diffeomorphisms}

As an example, we apply here the characterization of nonuniform
hyperbolicity and give sufficient conditions for a map on the boundary of
the set of Anosov diffeomorphisms to be (nonuniformly)\ hyperbolic.

Assume $f_{m}$ is a sequence of Anosov diffeomorphisms on $M={{{{\mathbb{R}}}%
}}^{d}/{{{{\mathbb{Z}}}}}^{d}$. Let $f$ be a diffeomorphism on $M$ with an
ergodic measure $\mu $. Assume $A_{m}$ is an increasing sequence of open
sets, $A=\bigcup_{m=1}^{\infty }A_{m}$, and suppose $A$ is $f$-invariant. We
denote the uniform hyperbolic coefficients of $f_{m}$ by $c_{m},\lambda _{m}$
(as in the definition) and $a_{m}$ (the bound on the angle between $E^{u}$
and $E^{s}$ as in (\ref{r:cangle})).

We assume $f_{m}$ converges to $f$ in the following sense:

(i) The functions $f_{m}$ and $f$ coincide on $A_{m}$ (hence also $Df_{m}$
and $Df$ coincide on $A_{m}$),

(ii)\ There is an uniform bound $b>0$ on $|Df_{m}|$ for all $m$ and on $|Df|$%
,

(iii)\ For all $m$, $\lambda _{m}\leq \lambda $ for some constant $\lambda
<1 $,

(iv) $\mu (A)=1$.

We can interpret these conditions as slowing down a uniformly hyperbolic
diffeomorphism (or perturbing it in another way to destroy uniformity)\ on a
set $A^{c}$ of measure $0$.

\begin{theorem}
\label{t:anosov}Assume $f_{m}$ is a sequence of Anosov diffeomorphisms
converging pointwise to $f$ on the set $A$, satisfying all of the above. If%
\begin{equation*}
\lim_{m\rightarrow \infty }\mu (A_{m}^{c})\log (a_{m}c_{m})=0,
\end{equation*}%
then $f$ is (nonuniformly)\ hyperbolic.
\end{theorem}

One can show that these conditions are satisfied, for example, in the case
of Katok's construction \cite{Katok:79} of a nonuniformly hyperbolic
diffeomorphism as a slowed down Arnold cat-map in a neighborhood of the
fixed point $0$.

Prior to proving the Theorem, we develop another tool based on the previous
section for showing that a diffeomorphism is nonuniformly hyperbolic. Recall
that for Banach spaces, a continuous linear operator $\Gamma $ is invertible
if there exist operators $A,B$ and $0\leq \lambda <1$ such that $||I-A\Gamma
||\leq \lambda $, $||I-\Gamma B||\leq \lambda $ ($A,B$ are approximate left
and right inverses of $\Gamma $; this was used for example in the proof of
Proposition \ref{p:uh2}). We will need an analogue of this for Fr\'{e}chet
spaces:

\begin{proposition}
\label{p:frechet}Assume $\Gamma $ is a continuous, linear operator on a Fr%
\'{e}chet space $\mathcal{F}$. Assume there exists two families $V_{n}$, $%
U_{n}$, $n\in \boldsymbol{N}$, of neighborhoods of $0$ such that $\{\delta
U_{n}$, $n\in \boldsymbol{N}$, $\delta >0\}$ is a local base. If there exist
families of continuous operators $A_{n,\delta }$, $B_{n,\delta }$ on $%
\mathcal{F}$ such that for all $n\in \boldsymbol{N}$, $\delta >0$%
\begin{eqnarray}
x &\in &V_{n}\text{ }\Rightarrow (I-A_{n,\delta }\Gamma )x\in \delta U_{n}%
\text{,}  \label{r:left} \\
x &\in &V_{n}\text{ }\Rightarrow (I-\Gamma B_{n,\delta })x\in \delta U_{n}%
\text{,}  \label{r:right}
\end{eqnarray}%
then $\Gamma $ is bijective with a continuous inverse.
\end{proposition}

\begin{proof}
Assume first $\Gamma $ is not injective, and let $\Gamma y=0$ for some $%
y\not=0$. We can find $\varepsilon ,\delta >0$ small enough so that $%
\varepsilon y\in V_{1}$ and $\varepsilon y\not\in \delta U_{n}$, which
contradicts (\ref{r:left}). We now show the image of $\Gamma $ is
second-countable in $\mathcal{F}$, by showing that each open set in $%
\mathcal{F}$ contains uncountably many elements of $\Gamma (\mathcal{F)}$.
Choose any $y\in \mathcal{F}$, find $\varepsilon >0$ so that $\varepsilon
y\in V_{n}$, and then by (\ref{r:right}) $\Gamma (B_{n,\delta }(y))\in
y-\delta /\varepsilon U_{n}$. As $\{\delta U_{n}\}$ is a local base, each
neighborhood of $y$ contains an element $z\in \Gamma (\mathcal{F)}$. But
then each neighborhood of $y$ contains uncountably many elements of $\Gamma (%
\mathcal{F)}$ (i.e. $z$ multiplied with a small open interval around $1$).

As $\Gamma (\mathcal{F)}$ is second countable in $\mathcal{F}$, we apply the
Open mapping theorem as stated in \cite{Rudin:91}, Theorem 2.11, and deduce
that $\Gamma (\mathcal{F)=F}$ and that $\Gamma $ is an open map.
\end{proof}

\begin{corollary}
\label{c:nicechar}Let $f$ be a diffeomorphism on $M$ and $\mu $ an ergodic $%
f $-invariant Borel probability measure. Assume that for $\mu $-a.e. $x\in
M, $ there exist families of continuous operators $A_{n,\delta }$, $%
B_{n,\delta }$ on ${{{{\mathcal{N}}}}}$ so that for infinitely many positive
integers $n$ and all sufficiently small $\delta >0$, 
\begin{eqnarray}
||I-A_{n,\delta }\Gamma _{o(x)}||_{2n,n} &<&\delta ,  \label{r:leftc} \\
||I-\Gamma _{o(x)}B_{n,\delta }||_{2n,n} &<&\delta .  \label{r:rightc}
\end{eqnarray}%
Then $f$ is nonuniformly hyperbolic.
\end{corollary}

\begin{proof}
By applying Proposition \ref{p:frechet} with $V_{n},U_{n}$ being the sets $%
||.||_{2n}<1$, $||.||_{n}<1$ respectively, we see that for $\mu $-a.e. $x\in
M$, $\Gamma _{o(x)}$ is invertible in $\mathcal{N}$ with a continuous
inverse. It now suffices to apply Theorem \ref{t:mainhyp}.
\end{proof}

We now construct a set $P\subseteq M$ of full measure so that, as we will
see later, $\Gamma _{o(x)}$ is invertible for $x\in P$. We rely for now only
on ergodic-theoretical arguments and the fact that $A_{m}\nearrow A$, $\mu
(A)=1$.

Let $P_{\varepsilon }$, $\varepsilon >0$, be the set of all $x\in M$ so that
there are infinitely many indices $m_{1}<m_{2}<...<m_{i}<...$ and integers $%
j_{i}\geq \varepsilon /(1-\mu (A_{m_{i}}))$ such that%
\begin{equation}
f^{-j_{i}}(x)\in A_{m_{i}},\text{ }f^{-j_{i}+1}(x)\in
A_{m_{i}},...,f^{j_{i}}(x)\in A_{m_{i}}.  \label{r:in}
\end{equation}

\begin{lemma}
\label{l:pe}Assuming all of the above, $\mu (P_{\varepsilon })\geq
1-4\varepsilon $.
\end{lemma}

\begin{proof}
We set $P_{j,m}$ to be the set of all $x\in M$ so that $f^{-j}(x)\in
A_{m},...,f^{j}(x)\in A_{m}$. By the Birkhoff ergodic theorem, for $\mu $%
-a.e. $x\in M$,%
\begin{equation}
\sum_{i=1}^{n}\frac{1}{n}\mathbf{1}_{P_{j,m}}(f^{i}(x))\rightarrow \mu
(P_{j,m})  \label{r:hopa1}
\end{equation}%
as $n\rightarrow \infty $. Also by the Birkhoff ergodic theorem, for $\mu $%
-a.e. $x\in M$, any $\delta >0$ and $n$ large enough, $%
f(x),f^{2}(x),...,f^{n}(x)$ is in $A_{m}^{c}$ at most $n(1-\mu
(A_{m})+\delta )$ times, hence for the same set $\mathbf{1}_{P_{j,m}}$ has
value $0$ at most%
\begin{equation}
(2j+1)n(1-\mu (A_{m})+\delta )+2j  \label{r:hopa2}
\end{equation}
times (otherwise is $1$). From (\ref{r:hopa1}), (\ref{r:hopa2}), and as $%
\delta >0$ is arbitrary, we get%
\begin{equation*}
\mu (P_{j,m})\geq 1-(2j+1)(1-\mu (A_{m}))\text{.}
\end{equation*}%
Without loss of generality assume $m$ is large enough so that there is an
odd integer $2j_{m}+1$,%
\begin{equation*}
\frac{2\varepsilon }{1-\mu (A_{m})}+1\leq 2j_{m}+1\leq \frac{4\varepsilon }{%
1-\mu (A_{m})}\text{.}
\end{equation*}

We get then $\mu (P_{j_{m},m})\geq 1-4\varepsilon $ and $j_{m}\geq
\varepsilon /(1-\mu (A_{m}))$. It is a simple measure-theoretical argument
to show that there is a measurable set $P_{\varepsilon }$, $\mu
(P_{\varepsilon })\geq 1-4\varepsilon $, so that each $x\in M$ is in
infinitely many $P_{j_{m},m}$.
\end{proof}

We set $P=\bigcup_{\varepsilon >0}$ $P_{\varepsilon }$. The set $P$ can be
interpreted as roughly the set of Birkhoff-regular points with respect to
the increasing family of sets $A_{m}$. For example, in the case of Katok's
perturbed cat-map \cite{Katok:79}, $P$ is indeed the entire $M$ with the
exception of $0$ and the stable and unstable manifolds of $0$, as in Katok's
construction.

We now complete the proof of Theorem \ref{t:anosov}.

\begin{proof}
Assume $x\in P_{\varepsilon }$ for some $\varepsilon >0$, where $%
P_{\varepsilon }$ as above and by Lemma \ref{l:pe}, $\mu (P_{\varepsilon
})\geq 1-4\varepsilon $. To show that $\Gamma _{o(x)}$ is invertible, we
will apply Corollary \ref{c:nicechar}. For that, we construct $A_{n,\delta }$%
, $B_{n,\delta }$ for given sufficiently large $n\in \boldsymbol{N}$ and
sufficiently small $\delta >0$ satisfying\ (\ref{r:leftc}) and (\ref%
{r:rightc}).

Assume $m_{i}$ is a sequence of indices for which (\ref{r:in}) holds, and 
\begin{equation}
j_{i}\geq \varepsilon /(1-\mu (A_{m_{i}})).  \label{r:relA}
\end{equation}%
For now we will drop the subscript $i$ (as it will be sufficient that $m$ is
large enough). Denote by $\Gamma =\Gamma _{o(x)}$, $\Gamma ^{(m)}=\Gamma
_{o(x)}$ the orbits with respect to respectively $f$, $f_{m}$. As $f_{m}$ is
uniformly hyperbolic, Proposition \ref{p:uh1} implies that $\Gamma ^{(m)}$
has an inverse $\Upsilon ^{(m)}$ which is a well defined operator on $X_{n}$
for all positive integers $n$, therefore also on $\mathcal{N}$. Choose $%
n_{0} $ large enough so that $\lambda \exp (1/n_{0})<1$. Then for any $n\geq
n_{0}$ the relation (\ref{r:h6}) implies that%
\begin{equation}
||\Upsilon ^{(m)}||_{n,n}\leq a_{m}c_{m}\frac{1+\lambda _{m}\exp (1/n)}{%
1-\lambda _{m}\exp (1/n)}\leq \frac{4a_{m}c_{m}}{1-\lambda \exp (1/n_{0})}%
\text{,}  \label{r:relB}
\end{equation}%
As orbits and tangent orbits of $x$ with respect to $f$ and $f_{m}$ coincide
for iterations $-j,-j+1,...,j$ for some $j$ satisfying (\ref{r:relA}), by
applying (\ref{r:50}) from Appendix and (\ref{r:relA}) we get for $n\geq
n_{0}$%
\begin{eqnarray}
||\Gamma -\Gamma ^{(m)}||_{2n,n} &\leq &\sup_{|k|\geq j+1}\exp \left( \frac{%
-|k+1|}{n}+\frac{|k|}{2n}\right) |Df_{m}(f_{m}^{k}(x))-Df(f^{k}(x))|\leq 
\notag \\
&\leq &2b\exp \left( \frac{-j+1}{2n}\right) \leq 2b\exp \left( \frac{%
-\varepsilon }{2n\mu (A_{m}^{c})}+\frac{1}{2n_{0}}\right) .  \label{r:relC}
\end{eqnarray}

We now look for sufficient conditions for%
\begin{eqnarray}
||\Upsilon ^{(m)}||_{n,n}||\Gamma -\Gamma ^{(m)}||_{2n,n} &<&\delta ,
\label{r:relD} \\
||\Gamma -\Gamma ^{(m)}||_{2n,n}||\Upsilon ^{(m)}||_{2n,2n} &<&\delta 
\label{r:relD1}
\end{eqnarray}%
to hold. Let $c$ be the constant%
\begin{equation*}
c=\log 2b+\log \left( \frac{4}{1-\lambda \exp (1/n_{0})}\right) +\frac{1}{%
2n_{0}}\text{.}
\end{equation*}%
Then from (\ref{r:relB}) and (\ref{r:relC})\ one deduces that%
\begin{equation}
\mu (A_{m}^{c})\log a_{m}c_{m}+c\mu (A_{m}^{c})+\mu (A_{m}^{c})\log \frac{1}{%
\delta }<\frac{\varepsilon }{2n}  \label{r:relE}
\end{equation}%
implies (\ref{r:relD}), (\ref{r:relD1}). As $\mu (A_{m}^{c})\log
a_{m}c_{m}\rightarrow 0$ and $\mu (A_{m}^{c})\rightarrow 0$, we can find
sufficiently large $m=m(n,\delta ,\varepsilon )$ such that (\ref{r:relE})
holds. From (\ref{r:relD}), (\ref{r:relD1}) we deduce that $A_{n,\delta
}=B_{n,\delta }:=\Upsilon ^{(m)}$ satisfies (\ref{r:leftc}), (\ref{r:rightc}%
), as required for Corollary \ref{c:nicechar} to hold.
\end{proof}

Perhaps the best explanation of the difference between uniform and
nonuniform hyperbolicity in this context is the relation (\ref{r:relC}).
Here we showed that the difference of two operators in the norm $%
||.||_{2n,n} $ can be very small if two orbits are only locally close. For
uniform hyperbolicity and the norm $||.||_{n,n}$ to be small, orbits of two
points would have to be uniformly close.

\section{Shadowing Lemma}

In this section we assume that $f$ is a $C^{1+\alpha }$ diffeomorphism on $M$%
, $\alpha >0$. Recall that $\boldsymbol{y}\in M^{{{{{\mathbb{Z}}}}}}$ is a $%
\beta $-pseudoorbit if for each $k\in {{{{\mathbb{Z}}}}}$, $%
|f(y_{k-1})-y_{k}|<\beta $. We now define precisely the notion of a
shadowable invariant set and measure, used in the statement of Theorem \ref%
{t:shadowing}.

We say that $\mathbf{x}\in M^{{{{{\mathbb{Z}}}}}}$ $\varepsilon $-shadows $%
\mathbf{y}\in M^{{{{{\mathbb{Z}}}}}}$, if for all $k$, $|x_{k}-y_{k}|<%
\varepsilon $.

\begin{definition}
We say that a set $\Lambda $ is shadowable, if there exists an increasing
sequence of sets $\Lambda _{k}$, $\tbigcup_{k=1}^{\infty }\Lambda
_{k}=\Lambda $, and a number $\delta >0$ depending on $\Lambda _{k}$, such
that the following holds: For each $\rho >0$ small enough there exists $%
\beta >0$, $\beta =\beta (\rho ,\Lambda _{k}$,$\delta )$, such that if $%
\boldsymbol{y}\in M^{{{{{\mathbb{Z}}}}}}$ is a $\beta $-pseudoorbit, $y_{j}$
in $\delta $-neighborhood of $\Lambda _{k}$ for all $j$ then there exists $%
x\in M$ such that $o(x)$ $\rho $-shadows $\boldsymbol{y}$.

An invariant measure is shadowable if there exists a shadowable set of full
measure.
\end{definition}

Uniformly hyperbolic sets are shadowable with $\Lambda _{k}=\Lambda $ (\cite%
{Katok:95}). Katok proved that nonuniformly hyperbolic measures are
shadowable (\cite{Katok:95}, Theorem S.4.14), where $\Lambda _{k}$ are
locally uniformly hyperbolic components of the Pesin set. Note that in the
nonuniformly hyperbolic case, as well as in our more general setting, $%
\Lambda _{k}$ are typically not invariant.

We now prove Theorem \ref{t:shadowing} in several steps. We fix now $n\in {{{%
{\mathbb{N}}}}}\cup \{\infty \}$ and the space $X_{n}$. Assume as in the
statement of Theorem \ref{t:shadowing} that for each $x\in \Lambda $, $%
\Gamma _{o(x)}$ has a continuous inverse $\Upsilon _{o(x)}$ acting on $X_{n}$%
. Let $\Lambda _{m}^{\ast }$ be the set of all $x\in \Lambda $ such that 
\begin{equation}
||\Upsilon _{o(x)}||_{n}\leq m{{{\text{.}}}}  \label{r:lambdam}
\end{equation}

The following Proposition is the key in the proof of Theorem \ref%
{t:shadowing}. In it we find $\delta $ small enough, construct the grading $%
\Lambda _{k}$ and show that the operator $\Gamma _{\boldsymbol{y}}$ is
invertible in $X_{\infty }$ for $\boldsymbol{y}$ being a $\delta $%
-pseudoorbit close enough to $\Lambda _{k}$. More precisely:

\begin{proposition}
\label{p:delta}Assume $m,n$ are positive integers, and let $\Lambda
_{m}^{\ast }$ be as defined by (\ref{r:lambdam}). Then $\Lambda _{m}^{\ast }$
can be decomposed into an increasing union of sets $\Lambda _{m}^{\ast
}=\tbigcup_{r=1}^{\infty }\Lambda _{m,r}$ such that the following holds:\
for any positive integer $r$ there exists $\delta >0$ (depending on $m,r$)
such that if $\boldsymbol{y}$ is a $\delta $-pseudoorbit, $y_{j}$ in a $%
\delta $-neighborhood of $\Lambda _{m,r}$ for all $j\in {{{{\mathbb{Z}}}}}$,
then $\Gamma _{\boldsymbol{y}}$ has a continuous inverse $\Theta $ in $%
X_{\infty }$, such that 
\begin{equation}
||\Theta ||_{\infty }\leq 4md^{3}/(1-\func{exp}(-1/n)){{{\text{.}}}}
\label{p:bnd}
\end{equation}
\end{proposition}

We first outline the proof of Proposition \ref{p:delta}. First we construct
a countable decomposition of $\Lambda _{m}^{\ast }$ into sets $\Lambda
_{m,r} $ so that the inverses of\ operators $\Gamma $ for two points in $%
\Lambda _{m,r}$ are close in some sense. Then we show in two steps that for $%
\delta $ small enough, $\Gamma _{{{{\mathbf{y}}}}}$ has an approximate left
and an approximate right inverse, if ${{{\mathbf{y}}}}$ is a $\delta $%
-pseudoorbit and $y_{j}$ in a $\delta $-neighborhood of $\Lambda _{m,r}$.
From this we deduce that $\Gamma _{{{{\mathbf{y}}}}}$ is invertible in both $%
X_{n}$ and $X_{\infty }$.

First note that for any $x\in M$ and\ any $X_{n}$, $1\leq n\leq \infty $,
the operator $\Gamma _{o(x)}$ is bounded and quasi-diagonal in the sense
that the diagonal elements are identity operators, and the only other
non-vanishing elements in its matrix representation are on the lower
diagonal (a precise definition is in the Appendix to the paper).

Let $\Upsilon $ be the inverse of $\Gamma _{o(x)}$. By Lemma \ref{l:matrix}
in the Appendix and the comment after the Lemma, $\Upsilon $ has a matrix
representation $(\Upsilon _{i,j})$ (its elements $\Upsilon _{i,j}$ are
linear operators on ${{{{\mathbb{R}}}}}^{d}$).

\begin{lemma}
\label{l:eps}Let $m,p$ be positive integers and $\varepsilon >0$, and let $%
\Lambda _{m}^{\ast }$ be as defined by (\ref{r:lambdam}). Then $\Lambda
_{m}^{\ast }$ can be decomposed into an increasing union of sets $\Lambda
_{m}^{\ast }=\tbigcup_{r=1}^{\infty }\Lambda _{m,r}$, such that if $z,%
\widetilde{z}\in \Lambda _{m,r}$, if $\Upsilon ,\widetilde{\Upsilon }$ are
the inverses of $\Gamma _{o(z)},\Gamma _{o(\widetilde{z})}$ in $X_{n}$ and $%
|z-\widetilde{z}|\leq 1/r$, then 
\begin{equation}
\sup_{|j|\leq p}|\Upsilon _{0,j}-\widetilde{\Upsilon }_{0,j}|<\varepsilon {%
\text{.}}  \label{r:taueps}
\end{equation}
\end{lemma}

\begin{proof}
Assume $z_{k}\in \Lambda _{m}$ for all integers $k$. First note that if $%
z_{k}\rightarrow z$ as $k\rightarrow \infty $ in $M$, then $\Gamma
_{o(z_{k})}$ converges to $\Gamma _{o(z)}$ pointwise (i.e. for each matrix
element $i,j$). Say $\Upsilon ^{(k)},\Upsilon $ are the inverses of $\Gamma
_{o(z_{k})}$,$\Gamma _{o(z)}$ in $X_{n}$. As $\Upsilon _{i,j}^{(k)},\Upsilon
_{i,j}$ are uniformly bounded for a given $i,j$ (see (\ref{r:exp2}) in the
Appendix), $\Upsilon ^{(k)}$ converges to $\Upsilon $ pointwise since the
inverse of $X_{n}$ is unique. Reasoning by contradiction, we find for each $%
z\in \Lambda _{m}$ a $\delta _{z}$-neighborhood, $\delta _{z}>0$ so that if $%
|z-\widetilde{z}|<\delta _{z}$, (\ref{r:taueps}) holds. Now $\Lambda _{m,r}$
is the set of all $z\in \Lambda _{m}^{\ast }$ such that $\delta _{z}<1/r$.
\end{proof}

We now introduce the notation and write explicitly the relations equivalent
to invertibility in $X_{n}$. Choose any sequence $z_{k}\in \Lambda _{m}$, $%
k\in {{{{\mathbb{Z}}}}}$, and denote by $\Upsilon ^{(k)}$ the inverse of $%
\Gamma _{o(z_{k})}$ in $X_{n}$, which then by definition satisfies 
\begin{eqnarray}
\Upsilon ^{(k)}\Gamma _{o(z_{k})} &=&\Gamma _{o(z_{k})}\Upsilon ^{(k)}=I,
\label{r:17a} \\
||\Upsilon ^{(k)}||_{n} &\leq &m\text{.}  \label{r:17b}
\end{eqnarray}%
If $(\Upsilon _{i,j}^{(k)})$, $i,j\in {{{{\mathbb{Z}}}}}$, is the matrix
representation of $\Upsilon ^{(k)}$ (which exists because of Lemma \ref%
{l:matrix} in the Appendix), then (\ref{r:17a}) and the definition of $%
\Gamma _{o(z_{k})}$ imply that for all $\,i,j\in {{{{\mathbb{Z}}}}}$, 
\begin{eqnarray}
\Upsilon _{0,j-i}^{(i)}-\Upsilon _{0,j-i+1}^{(i)}Df(z_{i,j-i}) &=&\delta
_{j-i}I{{{\text{,}}}}  \label{r:in1} \\
-Df(z_{i-1,0})\Upsilon _{0,j-i+1}^{(i-1)}+\Upsilon _{1,j-i+1}^{(i-1)}
&=&\delta _{j-i}I\text{{.}}  \label{r:in15}
\end{eqnarray}%
where $z_{i,j}=f^{j}(z_{i})$, $\delta _{j}$ is the Kronecker symbol $\delta
_{0}=1$, $\delta _{j}=0$ for $j\not=0$, and $I$ in (\ref{r:in1}), (\ref%
{r:in1}) is the identity operator on ${{{{\mathbb{R}}}}}^{d}$. Furthermore, (%
\ref{r:17b}) implies (using (\ref{r:44})) for rows $i=0,1$ that%
\begin{eqnarray}
\sum_{j\in {{{{\mathbb{Z}}}}}}\func{exp}(|j|/n)|\Upsilon _{0,j}^{(i)}| &\leq
&md\sqrt{d}{{{\text{,}}}}  \label{r:in2b} \\
\sum_{j\in {{{{\mathbb{Z}}}}}}\func{exp}(|j|/n)|\Upsilon _{1,j}^{(i)}| &\leq
&m\func{exp}(1/n)d\sqrt{d}\text{{.}}  \label{r:in25b}
\end{eqnarray}

Let $c_{1}$, $c_{2}$ be the constants related to the continuity and H\"{o}%
lder continuity of $Df$ on $M$, i.e. such that for all $z_{1},z_{2}\in M$, 
\begin{eqnarray}
|Df(z_{1})| &\leq &c_{1}{{{\text{,}}}}  \label{r:in3} \\
|Df(z_{1})-Df(z_{2})| &\leq &c_{2}|z_{1}-z_{2}|^{\alpha }{{{\text{.}}}}
\label{r:in4}
\end{eqnarray}

\begin{lemma}
\label{l:step1}Assume $m$ is an integer and $\Lambda _{m}^{\ast }$ as
defined by (\ref{r:lambdam}). Then there exists $\delta >0$ such that if $%
\boldsymbol{y}$ is a $\delta $-pseudoorbit, $y_{j}$ in a $\delta $%
-neighborhood of $\Lambda _{m}^{\ast }$, then $\Gamma _{\mathbf{y}}$ has an
approximate left inverse $\tilde{\Theta}$ in $X_{n}$, that means a
continuous operator $\tilde{\Theta}$ such that $||\tilde{\Theta}\Gamma _{%
\mathbf{y}}-I||_{n}\leq 1/2$. Furthermore, $||\tilde{\Theta}||_{n}\leq md%
\sqrt{d}{{{\text{.}}}}$
\end{lemma}

\begin{proof}
Let $z_{i}\in \Lambda _{m}^{\ast }$ such that $|z_{i}-y_{i}|<\delta $ and
let $\Upsilon ^{(i)}$ be the inverse of $\Gamma _{o(z_{i})}$ in $X_{n}$. We
define%
\begin{equation}
\tilde{\Theta}_{i,j}=\lambda ^{|j-i|}\Upsilon _{0,j-i}^{(i)}  \label{r:in5}
\end{equation}%
for some $0<\lambda <1$ to be defined later. We denote by $\Delta =\tilde{%
\Theta}\Gamma _{{\mathbf{y}}}-I$, and then by definition of $\Gamma _{{{{%
\mathbf{y}}}}}$ and since $\delta _{j-i}=\lambda ^{|j-i|}\delta _{j-i}$, 
\begin{equation*}
\Delta _{i,j}=\tilde{\Theta}_{i,j}-\tilde{\Theta}_{i,j+1}Df(y_{j})-\lambda
^{|j-i|}\delta _{j-i}I{\text{{.}}}
\end{equation*}%
Substituting $\delta _{j-i}I$ with (\ref{r:in1}), applying (\ref{r:in5}) and
then (\ref{r:in3}) and (\ref{r:in4}) we get 
\begin{eqnarray*}
|\Delta _{i,j}| &=&|\lambda ^{|j-i|}\Upsilon
_{0,j-i+1}^{(i)}Df(z_{i,j-i})-\lambda ^{|j-i+1|}\Upsilon
_{0,j-i+1}^{(i)}Df(y_{j})|\leq \\
&\leq &\lambda ^{|j-i|}|\Upsilon _{0,j-i+1}^{(i)}|\cdot \min
\{2c_{1},c_{2}|z_{i,j-i}-y_{j}|^{\alpha }\}+ \\
&&+\lambda ^{|j-i|}(1-\lambda )\cdot c_{1}|\Upsilon _{0,j-i+1}^{(i)}|{\text{.%
}}
\end{eqnarray*}

For some integer $q$ also to be defined later, we can rewrite that as 
\begin{equation}
|\Delta _{i,j}|\leq \left\{ 
\begin{array}{cc}
|\Upsilon _{0,j-i+1}^{(i)}|\left( 2c_{1}\lambda ^{q}+c_{1}(1-\lambda )\right)
& |j-i|\geq q \\ 
|\Upsilon _{0,j-i+1}^{(i)}|\left( c_{2}d(q)+c_{1}(1-\lambda )\right) & 
|j-i|<q,%
\end{array}%
\right.  \label{r:in6}
\end{equation}%
where 
\begin{equation*}
d(q)=\sup_{|j-i|<q}|z_{i,j-i}-y_{j}|^{\alpha }{\text{.}}
\end{equation*}

Now we bound $||\Delta ||_{n}$ using (\ref{r:40}). From (\ref{r:in2b}) and (%
\ref{r:in6}) we get that for all $i$, 
\begin{equation}
\func{exp}(-|i|/n)\sum_{j}\func{exp}(|j|/n)|\Delta _{i,j}|\leq \lambda
^{q}2c_{1}c_{3}+c_{1}(1-\lambda )c_{3}+c_{2}d(q)c_{3},  \label{r:in7b}
\end{equation}%
where $c_{3}=\exp (1/n)md\sqrt{d}$. We can now choose $0<\lambda <1$
(depending only on $c_{1},n,m)$ so that the second summand in (\ref{r:in7b})
is $\leq 1/8$, and an integer $q$ large enough (also depending only on $%
c_{1},n,m$) so that the first summand is $\leq 1/8$. We can also find $%
\delta >0$ small enough (depending only on $c_{2},n,m,\alpha $) so that $%
d(q) $ is small enough and that the third summand is also $\leq 1/8$, hence $%
||\Delta ||_{n}\leq 3/8<1/2$.

The relation 
\begin{equation*}
||\tilde{\Theta}||_{n}\leq md\sqrt{d}{{{\text{.}}}}
\end{equation*}%
follows directly from (\ref{r:17b}), (\ref{r:in5}), (\ref{r:40}) and (\ref%
{r:44}).
\end{proof}

In the following Lemma we again use the notation $\tilde{\Theta}$, $\lambda $
and $\delta $. They are not necessarily the same as in the statement and the
proof of Lemma \ref{l:step1}, but the notation is kept for simplicity.

\begin{lemma}
\label{l:step2}Assume $m$ is an integer. Then there exists $p\in \boldsymbol{%
N}$ such that for each $r\in \boldsymbol{N}$, if $\Lambda _{m,r}$ is as
constructed in Lemma \ref{l:eps} then the following holds:\ there exists $%
\delta >0$, such that if $\boldsymbol{y}$ is a $\delta $-pseudoorbit, $y_{j}$
in a $\delta $-neighborhood of $\Lambda _{m,r}$, then $\Gamma _{\boldsymbol{y%
}}$ has an approximate right inverse $\tilde{\Theta}$ in $X_{n}$, that means
a continuous operator $\tilde{\Theta}$ such that $||\Gamma _{y}\tilde{\Theta}%
-I||_{n}\leq 1/2$. Furthermore, $||\tilde{\Theta}||_{n}\leq md\sqrt{d}{{{%
\text{.}}}}$
\end{lemma}

\begin{proof}
Assume $z_{i}\in \Lambda _{m}^{\ast }$ such that $|z_{i}-y_{i}|<\delta $ and
let $\Upsilon ^{(i)}$ be the inverse of $\Gamma _{o(z_{i})}$ in $X_{n}$. Let 
\begin{equation}
\tilde{\Theta}_{i,j}=\lambda ^{|j-i|}\Upsilon _{0,j-i}^{(i)}  \label{p:46}
\end{equation}%
for some $0<\lambda <1$ to be defined later, and let $\widetilde{\Delta }%
=\Gamma _{{\mathbf{y}}}\tilde{\Theta}-I$. Then 
\begin{equation*}
\widetilde{\Delta }_{i,j}=-Df(y_{i-1})\tilde{\Theta}_{i-1,j}+\tilde{\Theta}%
_{i,j}-\lambda ^{|j-i|}\delta _{j-i}I{\text{{.}}}
\end{equation*}

Similarly as in Lemma \ref{l:step1}, using (\ref{p:46}), (\ref{r:in15}) and
then (\ref{r:in3}) and (\ref{r:in4}), we get 
\begin{eqnarray*}
|\widetilde{\Delta }_{i,j}| &\leq &\lambda ^{|j-i|}|\Upsilon
_{0,j-i+1}^{(i-1)}|\cdot \min \{2c_{1},c_{2}|z_{i-1}-y_{i-1}|^{\alpha }\}+ \\
&&+\lambda ^{|j-i|}(1-\lambda )\cdot c_{1}|\Upsilon
_{0,j-i+1}^{(i-1)}|+\lambda ^{|j-i|}|\Upsilon _{0,j-i}^{(i)}-\Upsilon
_{1,j-i+1}^{(i-1)}|{\text{.}}
\end{eqnarray*}%
For some positive integer $p$ to be chosen later, we deduce that 
\begin{equation}
|\widetilde{\Delta }_{i,j}|\leq \left\{ 
\begin{array}{cc}
|\Upsilon _{0,j-i+1}^{(i)}|\left( 2c_{1}\lambda ^{p}+c_{1}(1-\lambda
)\right) +\lambda ^{p}|\Upsilon _{0,j-i}^{(i)}|+\lambda ^{p}|\Upsilon
_{1,j-i+1}^{(i-1)}| & |j-i|\geq p \\ 
|\Upsilon _{0,j-i+1}^{(i)}|\left( c_{2}|z_{i-1}-y_{i-1}|^{\alpha
}+c_{1}(1-\lambda )\right) +e(p) & |j-i|<p,%
\end{array}%
\right.  \label{r:hum}
\end{equation}%
where 
\begin{equation*}
e(p)=\sup_{|j-i|\leq p}|\Upsilon _{0,j-i}^{(i)}-\Upsilon _{1,j-i+1}^{(i-1)}|{%
\text{.}}
\end{equation*}%
From (\ref{r:in2b}), (\ref{r:in25b})\ and (\ref{r:hum}) we get that for all $%
i$, 
\begin{eqnarray}
\func{exp}(-|i|/n)\sum_{j}\func{exp}(|j|/n)|\widetilde{\Delta }_{i,j}| &\leq
&\lambda ^{p}2c_{1}c_{3}+\lambda ^{p}2m+c_{1}(1-\lambda )c_{3}+  \notag \\
&&+c_{2}\sup_{i\in {{{{\mathbb{Z}}}}}}|z_{i-1}-y_{i-1}|^{\alpha
}c_{3}+c_{4}(p)e(p),  \label{r:pppp4}
\end{eqnarray}%
where $c_{3}=\exp (1/n)md\sqrt{d}$, $c_{4}(p)=\sum_{|j|<p}\func{exp}(|j|/n){%
\text{.}}$

Again $\,0<\lambda <1$ is chosen so that the third summand in (\ref{r:pppp4}%
) is $\leq 1/8$. We choose $p$ large enough so that the sum of the first two
summands is $\leq 1/8$, and we choose $\delta _{1}$ small enough so that if $%
y_{i}$ is $\delta _{1}$-close to $z_{i}$ for all integers $i$, the fourth
summand is $\leq 1/8$.

We now set $\varepsilon =1/8c_{4}(p)$, and apply Lemma \ref{l:eps} and find
for the chosen $p$ an increasing decomposition $\Lambda
_{m}=\tbigcup_{r=1}^{\infty }\Lambda _{m,r}$. Let $\Upsilon =\Upsilon ^{(i)}$%
, $\widetilde{\Upsilon }=S^{-1}\Upsilon ^{(i-1)}S\,$, $z=z_{i}$ and $\tilde{z%
}=f(z_{i-1})$ for any integer $i$. Then $e(p)$ is equal to the left-hand
side of (\ref{r:taueps}). Now there exists $\delta \leq \delta _{1}$ such
that if $\boldsymbol{y}$ is a $\delta $-pseudoorbit, $|y_{i}-z_{i}|<\delta $
for all integers $i$, then $|z-\tilde{z}|<1/r$ and $z,\tilde{z}\in \Lambda
_{m,r}$, hence $e(p)<\varepsilon $ and the fourth summand is $\leq 1/8$.
From (\ref{r:40})\ we deduce that $||\widetilde{\Delta }||_{n}\leq 1/2$.

The bound on $||\tilde{\Theta}||_{n}$ is obtained as in Lemma \ref{l:step1}.
\end{proof}

We now complete the proof of Proposition \ref{p:delta}.

\begin{proof}
Let $\Lambda _{m,r}$ be the sets constructed in Lemma \ref{l:step2}, and let 
$\delta >0$ be the smaller of the $\delta $'s constructed in Lemmas \ref%
{l:step1}, \ref{l:step2} for given positive integers $m,r$. We first show
that for that $\delta $, if ${{{\mathbf{y}}}}$ satisfies the conditions of
the Proposition, than $\Gamma _{{{{\mathbf{y}}}}}$ has a continuous inverse
in $X_{n}$. Let $\tilde{\Theta}$ be the approximate left inverse constructed
in Lemma \ref{l:step1}. As $||\tilde{\Theta}\Gamma _{y}-I||_{n}\leq 1/2$, $%
\tilde{\Theta}\Gamma _{y}$ has the inverse 
\begin{equation*}
(\tilde{\Theta}\Gamma _{{{{\mathbf{y}}}}})^{-1}=\sum_{k=0}^{\infty }(I-%
\tilde{\Theta}\Gamma _{y})^{k}{{{\text{,}}}}
\end{equation*}%
as the series is absolutely convergent in $X_{n}$. We deduce that $\Gamma _{{%
\mathbf{y}}}$ has the left inverse $\Theta =(\tilde{\Theta}\Gamma _{{{{%
\mathbf{y}}}}})^{-1}\tilde{\Theta}$ in $X_{n}$, with the norm 
\begin{equation*}
||\Theta ||_{n}\leq ||(\tilde{\Theta}\Gamma _{{{{\mathbf{y}}}}})^{-1}||_{n}||%
\tilde{\Theta}||_{n}\leq \frac{1}{1-||\tilde{\Theta}\Gamma _{y}-I||_{n}}%
\cdot md\sqrt{d}\leq 2md\sqrt{d}{{{\text{.}}}}
\end{equation*}

Similarly we show that $\Gamma _{{\mathbf{y}}}$ has a bounded right inverse,
hence $\Theta $ must be the inverse of $\Gamma _{{\mathbf{y}}}$ in $X_{n}$.
We finally show that $\Theta \in X_{\infty }$. This\ and (\ref{p:bnd})\
follow directly from Lemma \ref{l:uh3}.
\end{proof}

We now prove Theorem \ref{t:shadowing} by constructing a contraction mapping
on $X_{\infty }$.

\begin{proof}
By Proposition \ref{p:delta}, there exist constants $\delta >0$ and $K>0$
such that for every $\delta $-pseudoorbit $\boldsymbol{y}$, $y_{k}$ in $%
\delta $-neighborhood of $\Lambda _{m,r}$ for all $k$, operator $\Gamma _{%
\boldsymbol{y}}$ is invertible on $X_{\infty }$ and $||\Gamma _{\boldsymbol{y%
}}^{-1}||_{\infty }\leqslant K$. We define maps $A_{\boldsymbol{y}}$ and $%
\Phi _{\boldsymbol{y}}$ on $X_{\infty }$ as 
\begin{eqnarray*}
A_{\boldsymbol{y}}(\boldsymbol{\xi })_{n} &=&f(y_{n-1}+\xi _{n-1})-y_{n}, \\
\Phi _{\boldsymbol{y}}(\boldsymbol{\xi }) &=&\boldsymbol{\xi }+\Gamma _{%
\boldsymbol{y}}^{-1}(A_{\boldsymbol{y}}(\boldsymbol{\xi })-\boldsymbol{\xi }%
).
\end{eqnarray*}%
It is not hard to show that $A_{\boldsymbol{y}}$ is differentiable on a
neighborhood of $0$ in $X_{\infty }$. The derivative of $A_{\boldsymbol{y}}$
at $\boldsymbol{\xi }$ is the linear operator on $X_{\infty }$ given by 
\begin{equation*}
(DA_{\boldsymbol{y}}(\boldsymbol{\xi })\boldsymbol{\eta })_{n}=Df(y_{n-1}+%
\xi _{n-1})\eta _{n-1}.
\end{equation*}%
Take any $0<\kappa <1$. Since $Df$ is continuous and $DA_{\boldsymbol{y}%
}(0)=I-\Gamma _{\boldsymbol{y}}$, for any $\rho >0$ small enough (smaller
than some $\rho _{0})$, if $||\boldsymbol{\xi }||_{\infty }\leqslant \rho $
then%
\begin{equation}
||DA_{\boldsymbol{y}}(\boldsymbol{\xi })-(I-\Gamma _{\boldsymbol{y}%
})||_{\infty }\leqslant \frac{\kappa }{K}.  \label{p:con1}
\end{equation}%
We now choose $\beta =\beta (\rho ,\Lambda _{m,r},\delta )$, $\beta \leq
\delta $ so that 
\begin{equation}
K\beta \leqslant (1-\kappa )\rho  \label{p:con2}
\end{equation}%
and assume that $\boldsymbol{y}$ is a $\beta $-pseudoorbit, or equivalently%
\begin{equation}
||A_{\boldsymbol{y}}(0)||_{\infty }\leq \beta .  \label{p:con25}
\end{equation}%
For any $||\boldsymbol{\xi }||_{\infty }\leqslant \rho $ the map $\Phi _{%
\boldsymbol{y}}$ is differentiable at $\boldsymbol{\xi }$ and 
\begin{eqnarray*}
D\Phi _{\boldsymbol{y}}(\boldsymbol{\xi }) &=&I+\Gamma _{\boldsymbol{y}%
}^{-1}(DA_{\boldsymbol{y}}(\boldsymbol{\xi })-I) \\
&=&\Gamma _{\boldsymbol{y}}^{-1}(DA_{\boldsymbol{y}}(\boldsymbol{\xi }%
)-(I-\Gamma _{\boldsymbol{y}})).
\end{eqnarray*}%
Therefore for each $||\boldsymbol{\xi }||_{\infty }\leqslant \rho $, 
\begin{equation}
||D\Phi _{\boldsymbol{y}}(\boldsymbol{\xi })||_{\infty }\leqslant \kappa <1.
\label{p:con3}
\end{equation}%
Now let $||\boldsymbol{\xi }||_{\infty }\leqslant \rho $. Then by applying (%
\ref{p:con3}), (\ref{p:con25}) and finally (\ref{p:con2}) we deduce that 
\begin{eqnarray*}
||\Phi _{\boldsymbol{y}}(\boldsymbol{\xi })||_{\infty } &\leqslant &||\Phi _{%
\boldsymbol{y}}(\boldsymbol{\xi })-\Phi _{\boldsymbol{y}}(0)||_{\infty
}+||\Phi _{\boldsymbol{y}}(0)||_{\infty } \\
&\leqslant &\kappa \rho +K\beta \\
&\leqslant &\rho .
\end{eqnarray*}%
Thus we have proved that $\Phi _{\boldsymbol{y}}$ is a contraction on $\{||%
\boldsymbol{\xi }||_{\infty }\leqslant \rho \}$ and therefore has a unique
fixed point in $\{||\boldsymbol{\xi }||_{\infty }\leqslant \rho \}$. Since $%
\Phi _{\boldsymbol{y}}$ has the same fixed points as $A_{\boldsymbol{y}}$
and since the fixed points of $A_{\boldsymbol{y}}$ are exactly the orbits of 
$f$, we have proved that for $\Lambda _{m,r}$ we can find the required $%
\beta =\beta (\rho ,\Lambda _{m,r},\delta )$ as in the definition of
shadowable set. We now set $\Lambda _{k}=\cup _{m,r\leq k}\Lambda _{m,r}$
and $\beta (\rho ,\Lambda _{k},\delta )=\min_{m,r\leq k}\beta (\rho ,\Lambda
_{m,r},\delta )$ which completes the proof.
\end{proof}

One can further strengthen Theorem \ref{t:shadowing} and replace the
condition that $\Gamma _{o(x{{{\mathbf{)}}}}}$ is invertible in $X_{n}$ with 
$\Gamma _{o(x{{{\mathbf{)}}}}}$ having left and right inverses $\Upsilon
:X_{n_{1}}\rightarrow X_{n_{2}}$ for some $n_{1}\geq n_{2}$ (the proof is
essentially analogous but with more complex notation). By Theorem \ref%
{t:mainhyp}, the Shadowing lemma for nonuniformly hyperbolic measures as
established by Katok \cite{Katok:95} now follows as a corollary, while other
formulations in the nonuniformly hyperbolic case (\cite{Hirayama:84}, \cite%
{Pollicott:93}) can be established by further modifying the proof of Theorem %
\ref{t:shadowing}.

\section{Appendix:\ norms of operators on $l_{\infty }({{{{\mathbb{R}}}}}%
^{d})$}

\label{section:a}Here we summarize several technical results on the
operators on $l_{\infty }({{{{\mathbb{R}}}}}^{d})$ for the convenience of
the reader. We say a bounded linear operator on $l_{\infty }({{{{\mathbb{R}}}%
}}^{d})$ has a matrix representation, if there exist a family of linear
operators $A_{i,j}\in L({{{{\mathbb{R}}}}}^{d})$, $i,j\in {{{{\mathbb{Z}}}}}$%
, such that for each $\boldsymbol{x}\in l_{\infty }({{{{\mathbb{R}}}}}^{d})$%
, 
\begin{equation*}
(A\mathbf{x})_{i}\boldsymbol{=}\sum_{j}A_{i,j}x_{j}{\text{,}}
\end{equation*}%
where the series above absolutely converges for all $\boldsymbol{x}\in
l_{\infty }({{{{\mathbb{R}}}}}^{d})$. We note that one can construct bounded
linear operators on $l_{\infty }({{{{\mathbb{R}}}}}^{d})$ which have no
matrix representation. (Example for $d=1$:\ such an operator is $A\mathbf{x}%
=l(\mathbf{x)}\iota $, where $\iota \in l_{\infty }({{{{\mathbb{R}}}}})$, $%
\iota _{j}=1$ for all $j$, and $l(\mathbf{x)}$ is any continuous linear
functional defined so that $l(c\cdot \iota )=c$ for $c\in {{{{\mathbb{R}}}}}$%
, $l(\mathbf{x)}=0$ for all $\mathbf{x}$ with only finitely many non-zeros,
and extended by the Hahn-Banach theorem to the entire $l_{\infty }({{{{%
\mathbb{R}}}}}^{d})$. See e.g. \cite{Rudin:91}, Section 6, for a more
general discussion.)

Assume a bounded linear operator $A$ on $l_{\infty }({{{{\mathbb{R}}}}}^{d})$
has the matrix representation $A_{i,j}\in L({{{{\mathbb{R}}}}}^{d})$. Then
by triangle inequality, 
\begin{equation}
||A||_{\infty }\leq \sup_{i}\sum_{j}|A_{i,j}|{\text{.}}  \label{r:rel}
\end{equation}%
We say that an operator $A$ on\ $l_{\infty }({{{{\mathbb{R}}}}}^{d})$ is
quasi-diagonal, if there exist linear operators $A_{i}\in L({{{{\mathbb{R}}}}%
}^{d})$ so that 
\begin{equation}
(A\boldsymbol{x})_{i}=-A_{i-1}x_{i-1}+x_{i}{\text{,}}  \label{r:quasi}
\end{equation}%
hence a quasi-diagonal operator on $l_{\infty }({{{{\mathbb{R}}}}}^{d})$ is
bounded if and only if $|A_{i}|~$are bounded uniformly in $i$.

One can prove by choosing appropriate $\boldsymbol{x}\in l_{\infty }({{{{%
\mathbb{R}}}}}^{d})$ that there is equality in\ (\ref{r:rel}) if $d=1$. More
generally, we can deduce the following:

\begin{lemma}
Say a bounded linear operator $A$ on $l_{\infty }({{{{\mathbb{R}}}}}^{d})$
has matrix representation $A_{i,j}$. Then%
\begin{equation}
||A||_{\infty }\geq \frac{1}{d\sqrt{d}}\sup_{i}\sum_{j}|A_{i,j}|{\text{.}}
\label{r:rel2}
\end{equation}
\end{lemma}

\begin{proof}
Say $y_{1},...,y_{d}$ is the orthonormal basis of ${{{{\mathbb{R}}}}}^{d}$.
Then it is easy to show that for each vector $x\in {{{{\mathbb{R}}}}}^{d}$,
there is some $1\leq k\leq d$ such that $|(x,y_{k})|\geq |x|/\sqrt{d}$. Fix $%
i\in {{{{\mathbb{Z}}}}}$, and find vectors $z_{j}\in {{{{\mathbb{R}}}}}^{d},$
$j\in {{{{\mathbb{Z}}}}}$, $|z_{j}|=1$ such that $|A_{i,j}z_{j}|=|A_{i,j}|$.
We can now construct a partition of ${{{{\mathbb{Z}}}}}=V_{1}\cup ...\cup
V_{d}\,$\ such that for each $j\in V_{k}$, $|(A_{i,j}z_{j},y_{k})|\geq
|A_{i,j}|/\sqrt{d}$. Without loss of generality we can choose the sign of $%
z_{j}$ so that for each $j\in V_{k}$,%
\begin{equation*}
(A_{i,j}z_{j},y_{k})\geq |A_{i,j}|/\sqrt{d}.
\end{equation*}%
Now for non-empty $V_{k}$ we define $\boldsymbol{x}\in l_{\infty }({{{{%
\mathbb{R}}}}}^{d})$ with $x_{j}=z_{j}$ for $j\in V_{k}$, $x_{j}=0$
otherwise (and then $||\boldsymbol{x||}_{\infty }=1$). We calculate:%
\begin{eqnarray*}
||A||_{\infty } &\geq &|(A\boldsymbol{x})_{i}|=\left\vert \sum_{j\in
V_{k}}A_{i,j}z_{j}\right\vert \geq \left\vert \sum_{j\in
V_{k}}(A_{i,j}z_{j},y_{k})\right\vert = \\
&=&\sum_{j\in V_{k}}(A_{i,j}z_{j},y_{k})\geq \frac{1}{\sqrt{d}}\sum_{j\in
V_{k}}|A_{i,j}|\text{.}
\end{eqnarray*}%
We get the claim by summing that over all $k=1,...,d$.
\end{proof}

In cases interesting in this paper, one can show that the inverse of an
operator with a matrix representation also has a matrix representation:

\begin{lemma}
\label{l:matrix}Assume $A$ is a bounded quasi-diagonal operator on $%
l_{\infty }(\boldsymbol{R}^{d})$ represented with (\ref{r:quasi}) and with a
bounded inverse $B$. Then $B$ has a matrix representation $B_{i,j}$ such
that 
\begin{equation}
|B_{i,j}|\leq c\lambda ^{|j-i|}{\text{,}}  \label{r:exp}
\end{equation}%
where $c=a\lambda /(\lambda -ab(1-\lambda ))$, $a=\sup_{i}|A_{i}|$, $%
b=||B||_{\infty }$ and $\lambda $ any real number such that $%
ab/(ab+1)<\lambda <1$.
\end{lemma}

\begin{proof}
We fix $j\in {{{{\mathbb{Z}}}}}$ for now and choose any $\theta \in {{{{%
\mathbb{R}}}}}^{d}$. We define $\boldsymbol{y}$ as $y_{j}=\theta $, $y_{k}=0$
for $k\not=j$ and let $\boldsymbol{w}=B\boldsymbol{y}$, thus $\boldsymbol{y}%
=A\boldsymbol{w}$. We define linear operators $B_{i,j}$ with $B_{i,j}(\theta
)=w_{i}$ for any $\theta \in {{{{\mathbb{R}}}}}^{d}$ and $w_{i}$ defined as
above. As $B$ is linear and $B\boldsymbol{y}=\boldsymbol{w}$, so is $B_{i,j}$%
.

We now perturb linear operators $A$ around the $j$-th index, and define a
linear operator $A(z)$ for a real parameter $z\geq 1$ with 
\begin{equation*}
(A(z)\boldsymbol{x})_{i}=\QATOPD\{ . {-(1/z)A_{i-1}x_{i-1}+x_{i}\text{, }%
i\leq j}{-zA_{i-1}x_{i-1}+x_{i}\text{, }i\geq j+1.}
\end{equation*}%
Now clearly $A(1)=A$ and $||A(z)-A||_{\infty }\leq a(z-1)$. As $%
||A^{-1}||_{\infty }=b$, $A(z)$ is invertible for $z<1+1/ab$ and 
\begin{equation}
||A(z)^{-1}||\leq 1/(a^{-1}-b(z-1)){\text{.}}  \label{r:norm2}
\end{equation}%
Let $\boldsymbol{w}^{\ast }=A(z)^{-1}\boldsymbol{y}$. Then (\ref{r:norm2})
and the definition of $\boldsymbol{y}$ imply that%
\begin{equation*}
|w_{i}^{\ast }|\leq |\theta |/(a^{-1}-b(z-1)).
\end{equation*}

However it is easy to deduce from the definitions of $A(z),A$ (as in the
proof of Proposition \ref{p:uh2}) that $w_{i}^{\ast }=z^{|i-j|}w_{i}$, thus%
\begin{equation}
|w_{i}|\leq |\theta |z^{-|i-j|}/(a^{-1}-b(z-1)).  \label{r:norm3}
\end{equation}%
By setting $\lambda =1/z$ one gets the required bound on $|B_{i,j}|$ from
the right-hand side of (\ref{r:norm3}) and the definition of $B_{i,j}$.

If we now define an operator $B^{\ast }$ with $(B^{\ast }\boldsymbol{x}%
)_{i}=\sum_{j}B_{i,j}x_{j}$, then the series converges absolutely for any $%
\boldsymbol{x}\in l_{\infty }({{{{\mathbb{R}}}}}^{d})$. By calculating one
checks that $B^{\ast }$ is the inverse of $A$, so $B^{\ast }=B$ and $%
(B_{i,j})$ is its matrix representation.
\end{proof}

Now say $X_{n}$ are the Banach spaces defined in the introduction with the
norm $||\boldsymbol{x|}|_{n}=\sup_{k}\func{exp}(-|k|/n)|x_{k}|$. By a simple
isomorphism argument one gets that all the results above hold if we replace $%
A_{i,j}$ with 
\begin{equation}
\func{exp}((|j|-|i|)/n)A_{i,j}{\text{.}}  \label{r:rel3}
\end{equation}%
In particular (\ref{r:exp}) in $X_{n}$ becomes 
\begin{equation}
|B_{i,j}|\leq c\lambda ^{|j-i|}\func{exp}((|i|-|j|)/n){\text{.}}
\label{r:exp2}
\end{equation}

Furthermore, if $A$ is a linear operator on $X_{n}$ with the matrix
representation $A_{i,j}\in L({{{{\mathbb{R}}}}}^{d})$, then (\ref{r:rel})
and (\ref{r:rel2})\ imply that 
\begin{equation}
||A||_{n}\leq \sup_{i}\exp (-|i|/n)\sum_{j}\exp (|j|/n)|A_{i,j}|
\label{r:40}
\end{equation}%
and for all integers $i$,%
\begin{equation}
\exp (-|i|/n)\sum_{j}\exp (|j|/n)|A_{i,j}|\leq ||A||_{n}d\sqrt{d}\text{.}
\label{r:44}
\end{equation}

Similarly, if $A$ is a linear operator $A:X_{n}\rightarrow X_{m}$, then its
norm (denoted by $||.||_{n,m}$) can be bounded with%
\begin{equation}
||A||_{n,m}\leq \sup_{i}\exp (-|i|/m)\sum_{j}\exp (|j|/n)|A_{i,j}|.
\label{r:50}
\end{equation}

\end{document}